\documentclass[12pt]{article}
\usepackage{geometry}
\usepackage{amssymb}
\usepackage{amsmath}
\usepackage{amsthm}
\usepackage{thmtools}
\usepackage{mathtools}
\usepackage[scr=rsfs]{mathalpha}
\newtheorem{lemma}{Lemma}[section]
\newtheorem{corollary}[lemma]{Corollary}

\newtheorem{theorem}[lemma]{Theorem}
\theoremstyle{remark}
\newtheorem{remark}[lemma]{Remark}
\newtheorem{example}[lemma]{Example}
\theoremstyle{definition}
\newtheorem{definition}[lemma]{Definition}

\usepackage[utf8]{inputenc}
\usepackage{amsfonts}
\usepackage{array}
\usepackage{hyperref,cleveref}
\usepackage[english]{babel}
\usepackage[normalem]{ulem}
\usepackage{float}
\newcommand{\enstq}[2]{\left\{#1\mathrel{}\middle|\mathrel{}#2\right\}}
\newcommand{\norm}[1]{\left\|#1\right\|}
\newcommand{\N}{\mathbb{N}}

\newcommand{\R}{\mathbb{R}}
\newcommand{\nc}[1]{#1} % new concept

\newcommand{\abs}[1]{\left\lvert #1 \right\rvert}

\newcommand{\isdef}{\mathrel{\mathop:}=}

\newcommand\Includegraphics[2][]{%
	\raisebox{-.5\dimexpr\height-\ht\strutbox+\dp\strutbox}{\includegraphics[#1]{#2}}}
\usepackage{todonotes}
\usepackage{subcaption}

\makeatletter
\newsavebox{\@brx}
\newcommand{\llangle}[1][]{\savebox{\@brx}{\(\m@th{#1\langle}\)}%
	\mathopen{\copy\@brx\kern-0.5\wd\@brx\usebox{\@brx}}}
\newcommand{\rrangle}[1][]{\savebox{\@brx}{\(\m@th{#1\rangle}\)}%
	\mathclose{\copy\@brx\kern-0.5\wd\@brx\usebox{\@brx}}}
\makeatother

\usepackage{todonotes}
\usepackage{csquotes}
\makeatletter
\renewcommand*{\vec}[1]{\boldsymbol{#1}}
\makeatother
\usepackage{pseudocode}
\usepackage{algorithm}
\usepackage{mdframed}

\renewcommand{\#}[1]{\textup{Card}\left(#1\right)}
\makeatletter
\newenvironment{sampleCode}[2][plain]
{%
	\refstepcounter{pseudocode}%
	\ifthenelse{\equal{#1}{plain}}{\setboolean{pcode@plain}{true}}{\setboolean{pcode@plain}{false}}%
	\ifthenelse{\equal{#1}{ruled}}{\setboolean{pcode@ruled}{true}}{\setboolean{pcode@ruled}{false}}%
	\ifthenelse{\equal{#1}{display}}{\setboolean{pcode@disp}{true}}{\setboolean{pcode@disp}{false}}%
	\ifthenelse{\equal{#1}{shadowbox}}{\setboolean{pcode@shad}{true}}{\setboolean{pcode@shad}{false}}%
	\ifthenelse{\equal{#1}{doublebox}}{\setboolean{pcode@dbox}{true}}{\setboolean{pcode@dbox}{false}}%
	\ifthenelse{\equal{#1}{ovalbox}}{\setboolean{pcode@obox}{true}}{\setboolean{pcode@obox}{false}}%
	\ifthenelse{\equal{#1}{Ovalbox}}{\setboolean{pcode@Obox}{true}}{\setboolean{pcode@Obox}{false}}%
	\ifthenelse{\equal{#1}{framebox}}{\setboolean{pcode@fbox}{true}}{\setboolean{pcode@fbox}{false}}%
	\setcounter{pseudonum}{0}%
	\ifthenelse{\boolean{pcode@disp}}%
	{%
		\noindent\begin{math}%
		}%
		{%
			\begin{Sbox}%
				\begin{minipage}{\pcode@width}%
					\ifthenelse{\boolean{pcode@ruled}}
					{
						\noindent\rule{\pcode@width}{1pt}\hfill\\
						\pcode@AF{#2}
						\noindent\rule{\pcode@width}{1pt}\hfill\\[1ex]
					}
					{
						\pcode@AF{#2}
					}
					\noindent\begin{math}\begin{array}{@{\pcode@tab{1}}lr@{}}%
						}{}%
					}%
					{%
						\ifthenelse{\boolean{pcode@disp}}%
						{%
						\end{math}
					}%
					{%
						\ifthenelse{\boolean{pcode@ruled}}
						{
					\end{array}\end{math}
					\noindent\rule{\pcode@width}{1pt}\hfill
			\end{minipage}\end{Sbox}\noindent%
		}
		{\end{array}\end{math}\end{minipage}\end{Sbox}\noindent}%
\ifthenelse{\boolean{pcode@plain}}{\TheSbox}{}%
\ifthenelse{\boolean{pcode@ruled}}{\TheSbox}{}%
\ifthenelse{\boolean{pcode@shad}}{\shadowbox{\TheSbox}}{}%
\ifthenelse{\boolean{pcode@dbox}}{\doublebox{\TheSbox}}{}%
\ifthenelse{\boolean{pcode@obox}}{\cornersize*{4ex}\ovalbox{\TheSbox}}{}%
\ifthenelse{\boolean{pcode@Obox}}{\cornersize*{4ex}\Ovalbox{\TheSbox}}{}%
\ifthenelse{\boolean{pcode@fbox}}{\fbox{\TheSbox}}{}%
\bigskip%
}%
}%

\newcommand{\verteq}{\rotatebox{90}{$\,=$}}
\newcommand{\CLASS}[1]{\mbox{{\bfseries classdef} {\tt #1}}}
\newcommand{\METHODS}{\mbox{\bfseries methods}\\\begin{array}{@{\quad}lr@{}}}
\newcommand{\ENDMETHODS}{\end{array}\\\mbox{\bfseries end}}
\makeatother
\title{Fractured meshes}
\usepackage{geometry}

\usepackage{wrapfig}

\usepackage{pdfpages}
\date{}
\title{Fractured Meshes}

\author{
	M.~Averseng\thanks{Department of Mathematical Sciences, University of Bath, \tt ma3092@bath.ac.uk},
	\,\,
	X.~Claeys\thanks{Laboratoire Jacques-Louis Lions, Sorbonne Université, \tt xavier.claeys@upmc.fr}
	\,\,, R.~Hiptmair\thanks{Seminar for Applied Mathematics, ETH Zurich, \tt hiptmair@sam.math.ethz.ch}
}
\begin{document}
\maketitle
%\begin{frontmatter}



\begin{abstract}
This work introduces ``generalized meshes", a type of meshes suited for the discretization of partial differential equations in non-regular geometries. Generalized meshes extend regular simplicial meshes by allowing for overlapping elements and more flexible adjacency relations. They can have several distinct ``generalized" vertices (or edges, faces) that occupy the same geometric position. These generalized facets are the natural degrees of freedom for classical conforming spaces of discrete differential forms appearing in finite and boundary element applications. Special attention is devoted to the representation of fractured domains and their boundaries. An algorithm is proposed to construct the so-called {\em virtually inflated mesh}, which correspond to a ``two-sided" mesh of a fracture. Discrete $d$-differential forms on the virtually inflated mesh are characterized as the trace space of discrete $d$-differential forms in the surrounding volume.
\end{abstract}



\section*{Introduction} 

In an $n$-dimensional domain $\Omega \subset \R^n$, a fracture $\Gamma$ may be modeled by an embedded $(n-1)$-dimensional structure. For the finite element numerical simulation of partial differential equations (PDEs) in such domains, there are two approaches:
\begin{itemize}
\item[(i)] In the first one, the physical equations are approximated in the domain $\Omega$ surrounding the fracture, with boundary conditions on $\partial \Omega$ and $\Gamma$. A key example is the so-called \emph{mixed-dimensional modeling} of flow in fractured porous media. References on this topic include \cite[Section 5]{bookFracPorMed}, \cite[Section 3.5]{flowFracPor}, \cite{ahmed2017reduced,modelingFracASInterf,benchmark,celes2005compact,porepy,modelingFracAndBarr} but the list is far from being exhaustive.
\item[(ii)] In the second one, equivalent \emph{boundary integral equations} are formulated directly and only on the surface of $\Gamma$. This is the approach followed for example in wave scattering problems, where the obstacle $\Gamma$ is a \emph{complex screen} embedded in the space $\R^3$. A lot of attention has been devoted to the analysis and numerical resolution of this type of problem in recent years. References include \cite{bannister2022acoustic,buffaChris,chandler2021bem,quotientBem,multiscreens,cools2022preconditioners,precEFIE,yla2005surface}. 
\end{itemize}

We introduce a class of mathematical objects able to represent in a unified way the singular geometries appearing in both approaches. Those objects, that we call {\em generalized meshes}, are a simple generalization of {\em regular} meshes (meshes that represent regular manifolds). An $n$-dimensional generalized mesh is defined by a set of $n$-simplices, and an adjacency graph between them. The simplices are allowed to overlap arbitrarily, and their mutual adjacency is not dictated by shared vertices. 

Compared to other classical non-manifold geometric representations \cite{weiler1986topological,choi1989vertex,lee2001partial,rossignac1989sgc}, the generalized meshes defined here, while allowing to cover many practical cases, do not deal with the most general class of mixed-dimensional non-manifold geometries. For instance, they do not allow ``dangling edges". In counterpart, because of their simplicity, their storage is cheap and most operations -- e.g., assembly of a finite element matrix -- are performed in linear complexity. 

%Through generalized meshes, perspectives (i) and (ii) are connected via a {\em boundary} operator that maps an $n$-dimensional to an $(n-1)$-dimensional generalized mesh. 

Generalized meshes are adapted not only to describe singular geometries, but also to represent the classical discrete subspaces of the infinite dimensional spaces on which the PDEs act (with, for instance, conforming Galerkin discretization in mind). For those discrete subspaces, the degrees of freedom are directly related to specific geometric features of the mesh that we call {\em generalized subfacets}. The relation between generalized facets and the corresponding discrete subspaces is similar to how, in a regular mesh, continuous piecewise linear functions are related to vertices, Nédélec elements to edges, and so on. Importantly, a generalized mesh can have several distinct generalized vertices (or edges, or faces) located at the same geometric position. This allows, for instance, to discretize the two ``sides" of a surface independently. 
%
%
%Given a fracture $\Gamma$, represented by a (standard) mesh $\mathcal{M}_\Gamma$, the definition of the corresponding generalized mesh $\mathcal{M}^*_\Gamma$, which is the relevant object for BEM applications. The {\em virtual inflation} of a fracture, presented in \Cref{sec:fracturedMeshes}, is intuitively the procedure by which a singular surface $\Gamma$ in $\R^3$ is slightly ``inflated" to make it resemble an open set $\hat{\Gamma}\subset \R^3$. One can then analyze functions, or differential forms on $\Gamma$ with the help of the more regular topology of $\hat{\Gamma}$. When $\Gamma$ is a manifold, the virtual inflation coincides to the so-called {\em orientable covering} of a manifold, see e.g. \cite{hatcher2002algebraic}. 

This work is purely concerned with the combinatorial properties of generalized meshes. Our focus is on mathematical formalization rather than computational performance. This point of view seems valuable to us, as it paves the way for the rigorous discussion of the discretization of PDEs in non-regular geometries. In fact, one of the main motivation for this work originated in the need for a formal description of a boundary element method for scattering by a multi-screen \cite{averseng2022ddm,quotientBem,multiscreens}. 

\subsection*{Main results}
Our contributions are the following:\footnote{For notation, refer to later sections.}
\begin{itemize}
\item[(i)] We define generalized meshes, which contain regular simplicial meshes as a particular case.
\item[(ii)] We define the boundary $\partial \mathcal{M}^*$ of a generalized mesh $\mathcal{M}^*$, and prove that it coincides with the usual notion of boundary for a regular simplicial mesh.
\item[(iii)] We describe an algorithm, called {\em virtual inflation} (see also \cite{quotientBem}), to convert a usual simplicial mesh $\mathcal{M}_\Gamma$ of a complex surface $\Gamma$ into a generalized mesh $\mathcal{M}^*_\Gamma$, which, loosely speaking, represents the orientable two-sheeted covering of $\Gamma$ \cite[p. 234]{hatcher2002algebraic}. The key property of this algorithm is that $\mathcal{M}^*_\Gamma$ is the boundary (in the sense of (ii)) of the $3$-dimensional region surrounding $\Gamma$. 
\item[(iv)] We define finite dimensional spaces $\Lambda^d(\mathcal{M}^*)$ of discrete $d$-differential forms on a generalized mesh $\mathcal{M}^*$, generalizing the usual finite and boundary elements spaces on a simplicial mesh. Their Whitney basis functions are in bijection with the set of generalized $d$-facets of $\mathcal{M}^*$. The finite (or boundary) element matrices, corresponding to variational problems set in $\Lambda^d(\mathcal{M}^*)$, can be assembled with a simple and transparent algorithm. 
\item[(v)] Finally, we generalize the restriction operator $\textup{Tr}: \Lambda^d(\mathcal{M}^*) \to \Lambda^d(\partial \mathcal{M}^*)$, which maps a discrete $d$-differential form on a generalized mesh $\mathcal{M}^*$, to its restriction (or trace) on the boundary $\partial \mathcal{M}^*$. For a 3-dimensional fractured domain, we prove that $\textup{Tr}$ is surjective (provided that the fracture satisfies a simple connectivity condition when $d = 0$). 
\end{itemize}

\subsection*{Outline}

The paper is organized as follows: in \Cref{sec:simplices}, we review some combinatorial geometry background. We then define generalized meshes in \Cref{sec:genmeshes}, and their boundary in \Cref{sec:genbound}. In \Cref{sec:fracturedMeshes}, we focus on fractured meshes and their boundaries and describe the intrinsic virtual inflation. In \Cref{sec:FEEC}, we define spaces of discrete differential forms and study the surjectivity of the trace operator. A model Finite Element application is presented in \Cref{sec:numerics}.

\section*{List of notations}

\vspace{0.5cm}
\begin{footnotesize}

$\begin{array}{lclll}
\sigma_d(S)&& \textup{$d$-subsimplices of the simplex $S$}&\\
\mathcal{F}(S) && \textup{facets of the simplex $S$}&&\\
\sigma(S) && \textup{subsimplices of the simplex $S$, including $S$}&&\\
\abs{S} && \textup{convex hull of the simplex $S$} && Eq.~\eqref{eq:defHullS}\\
{[V_1,\ldots,V_{n+1}]}&& \textup{$n$-simplex oriented by the ordering $(V_1\,,\ldots\,,V_{n+1})$}&&\\
-[S] && \textup{simplex $S$ with the orientation opposite to $[S]$}&&\\
{[F]_{|[S]}} && \textup{facet $F$ of $S$ with the orientation induced by $[S]$} && Eq.~\eqref{inducedOrientFacet}\\
\sigma_d(\mathcal{M}) && \textup{$d$-subsimplices of the triangulation $\mathcal{M}$}&&Eq.~\eqref{eq:defSigmadM}\\
\sigma(\mathcal{M}) && \textup{subsimplices of the triangulation $\mathcal{M}$} && Eq.~\eqref{eq:defSigmadM}\\
\mathcal{F}(\mathcal{M}) && \textup{facets of the triangulation $\mathcal{M}$} && Eq.~\eqref{eq:defSigmadM}\\
\underset{\mathcal{M}}{\longleftrightarrow} && \textup{adjacency in the triangulation $\mathcal{M}$}\\
\textup{st}(S,\mathcal{M}) && \textup{star of the subsimplex $S$ of $\mathcal{M}$}&& Eq.~\eqref{eq:defStarM}\\
\partial \mathcal{M} && \textup{boundary of the triangulation $\mathcal{M}$}&&Eq.~\eqref{eq:defBoundM}\\	
\abs{\mathcal{M}} && \textup{geometry of the mesh $\mathcal{M}$}&&Eq.~\eqref{eq:defGeometry}\\
\mathcal{V}_{\mathcal{M}^*}\,,\,\,\mathcal{V} && \textup{vertex set of } \mathcal{M}^*&& $\Cref{defGenMesh}$\\
\mathbf{K}_{\mathcal{M}^*}\,,\,\,\mathbf{K}  && \textup{set of elements of } \mathcal{M}^*&& $\Cref{defGenMesh}$\\
\mathcal{K}_{\mathcal{M}^*}\,,\,\,\mathcal{K}  && \textup{realization function of } \mathcal{M}^*&& $\Cref{defGenMesh}$\\
\mathcal{G}_{\mathcal{M}^*}\,,\,\,\mathcal{G}  && \textup{adjacency graph of } \mathcal{M}^*&& $\Cref{defGenMesh}$\\
\mathbb{F}(\mathcal{M}^*) && \textup{split facets of $\mathcal{M}^*$}&&Eq.~\eqref{eq:defSplitFacet}\\
\overunderset{F}{\mathcal{M}^*}{\longleftrightarrow}\,,\,\,\overset{F}{\longleftrightarrow} && \textup{adjacency through $F$ in $\mathcal{M}^*$} && Eq.\eqref{eq:notationAdj}\\
\mathcal{N}_{\mathcal{M}^*}\,,\,\, \mathcal{N} && \textup{neighbor function of $\mathcal{M}^*$}&& Eq.~\eqref{def:Nmstar}\\
\sigma_d(\mathcal{M}^*) && \textup{subsimplices of $\mathcal{M}^*$} && $\Cref{def:subsimplMstar}$\\
{[\mathbf{k}]_{\mathcal{M}^*}} && \textup{orientation of the simplex attached to $\mathbf{k}$} && $\Cref{defOrientableGenMesh}$\\
&&\textup{assigned by the orientation of $\mathcal{M}^*$}&&\\
\textup{st}(S,\mathcal{M}^*) && \textup{generalized star of $S$ in $\mathcal{M}^*$} && Eq.~\eqref{eq:defGenStar}\\
\mathcal{G}_{\mathcal{M}^*}(S), \mathcal{G}(S) && \textup{graph between elements of $\textup{st}(S,\mathcal{M}^*)$} && $\Cref{fig:StarGraph}$\\
\mathbf{S}_d(\mathcal{M}^*)&&\textup{set of generalized $d$-subfacets of $\mathcal{M}^*$} && $\Cref{def:genSubfacet}$\\
\mathbf{F}(\mathcal{M}^*) &&\textup{generalized facets of $\mathcal{M}^*$} && Eq.~\eqref{eq:defFMstar}\\
\mathbf{F}_b(\mathcal{M}^*) && \textup{boundary split facets of $\mathcal{M}^*$} && Eq.~\eqref{eq:defFbMstar}\\
\partial \mathcal{M}^* && \textup{boundary of $\mathcal{M}^*$, temporarily denoted by $\partial^*\mathcal{M}^*$} && $\Cref{defGenBound}$\\
\mathcal{M}^*_{\Omega \setminus \Gamma} && \textup{fractured mesh} && $Section \ref{sec:defFracturedMesh}$\\
\mathcal{M}^*_\Gamma(\Omega) && \textup{extrinsic inflation of $\mathcal{M}_\Gamma$ via $\mathcal{M}_\Omega$} && $\Cref{def:inflByExt}$\\
\Theta(T_1,T_2) && \textup{geometric angle between $T_1$ and $T_2$} && Eq.~\eqref{eq:defGeomAngle}\\
\angle([T_1],T_2) && \textup{oriented angle between $[T_1]$ and $T_2$} && Eq.~\eqref{eq:defOrientedAngle}\\
\mathscr{F}_\Gamma && \textup{set of all orientations of the facets of $\mathcal{M}_\Gamma$} && Eq.~\eqref{eq:defOrientFacetsDouble}\\
\mathcal{M}^*_\Gamma  && \textup{intrinsic inflation of $\mathcal{M}_\Gamma$} && $\Cref{defMeshInflGeo}$\\
\lambda^K_V && \textup{barycentric coordinate} && Eq.\eqref{eq:defBarycentricFun}\\
\omega_{\mathbf{s}} && \textup{Whitney form associated to the generalized $d$-subfacet $\mathbf{s}$} &&$\Cref{def:WhitneyForms}$\\
\Lambda^d(\mathcal{M}^*)&&\textup{Space of discrete $d$-differential forms on $\mathcal{M}^*$}&&\\
\textup{Tr} && \textup{trace operator for Whitney forms}&& $\Cref{def:trace}$\\

\end{array}$

\end{footnotesize}
\section{Preliminaries}
\label{sec:simplices}

\subsection{Simplices and orientation}

An {\em $n$-simplex} $S$ is a subset of cardinality $(n+1)$ of some {\em vertex set} $\mathcal{V}$ (e.g. $\mathcal{V}= \R^3$). The elements of $S$ are called its {\em vertices}. For $d \in \{0,\ldots,n\}$, let \nc{$\sigma_d(S)$} be the set of $(d+1)$-subsets of $S$, also called \nc{$d$-subsimplices} of $S$. The $(n-1)$-subsimplices are called \nc{facets}; we write $\nc{\mathcal{F}(S)} \isdef \sigma_{n-1}(S)$.
The set of all subsimplices of $S$, including $S$, is denoted by $\sigma(S)$. An $n$-simplex is called a vertex, an edge, a triangle and a tetrahedron for $n = 0,1,2$ and $3$.

When the vertices of an $n$-simplex $S$ are points in $\R^m$ with $n \leq m$, we systematically assume that they are affinely independent. In this case, we denote by $\nc{\abs{S}}$ the closed convex hull of $S$:\footnote{In most references, a simplex is defined as the convex hull of its vertices, but the definition that we adopt here, (i.e. a simplex is defined as a finite set of vertices) is more convenient for our needs.}
\begin{small}
\begin{equation}
\abs{S} = \textup{Hull}(\{V_1\,,\ldots\,,V_{n+1}\}) \isdef \enstq{\sum_{i = 1}^{n+1} \lambda_i V_i}{\forall i \in \{1\,\ldots\,, n+1\}, \lambda_i \geq 0, \,\,\,\sum_{i = 1}^{n+1} \lambda_i = 1}\,. \label{eq:defHullS}
\end{equation}
\end{small}

For $n \geq 1$, an {\em orientation} of a simplex $S$ is an ordering $V_1\,,\ldots\,,V_{n+1}$ of its vertices, with the rule that two orderings define the same orientation if they differ by an even permutation. If $n=0$, i.e. when $S$ is a vertex, an orientation of $S$ is simply an element of $\{+,-\}$. 

An {\em oriented simplex} $[S]$ is a simplex $S$ together with a choice of orientation. As in \cite[Chap. 5]{lee2011introduction}, the simplex $\{V_1\,,\ldots\,,V_{n+1}\}$ oriented by the ordering $(V_1\,,\ldots\,,V_{n+1})$ is denoted by $[V_1\,,\ldots\,,V_n]$. If $[S]$ is an oriented simplex, then $-[S]$ refers to the same simplex with the opposite orientation. Hence for example
\[[\textup{A},\textup{B},\textup{C},\textup{D}] = -[\textup{A},\textup{C},\textup{B},\textup{D}] = [\textup{C},\textup{A},\textup{B},\textup{D}]\,.\]

When $n \geq 2$, an oriented $n$-simplex $[S]= [V_1,\ldots,V_{n+1}]$ {\em induces an orientation} on its facets: for the facet $F_i = \{V_1,\ldots,\widehat{V_i},\ldots,V_{n+1}\}$, where the hat denotes an omission, the induced orientation is
\begin{equation}
\label{inducedOrientFacet}
[F_i]_{|[S]} \isdef (-1)^{i+1}[V_1,\ldots,\widehat{V_i},\ldots,V_{n+1}]\,.
\end{equation}
For $n = 1$, given an oriented edge $E = [V_1,V_2]$, we adopt the convention
\[[V_1]_{|[E]} \isdef [V_1,-], \quad [V_2]_{|[E]} \isdef [V_2,+]\,.\] 
Let $n \geq 1$, and $S_1$, $S_2$ be two $n$-simplices sharing a facet $F$. Then we say that $[S_1]$ and $[S_2]$ are {\em consistently oriented} if
\begin{equation}
\label{eq:defConsistOrient}
[F]_{|[S_1]} = - [F]_{|[S_2]}
\end{equation}
i.e. they induce {\bf opposite} orientations on their common facet. If $\mathcal{V} \subset \R^n$, then the $n$-simplices have a {\em natural orientation}: $[V_1\,,\ldots\,,V_{n+1}]$ is naturally oriented if
\[\det \left[V_{2} - V_1\,,\ldots\,,V_{n+1} - V_1\right] > 0\,.\]
We record the following two elementary results for later use.
\begin{lemma}
\label{lemConsist}
If $[K]$ is an oriented $n$-simplex with $n \geq 2$, and $F$ and $F'$ are two distinct facets of $K$, then $[F]_{|[K]}$ and $[F']_{|[K]}$ are consistently oriented. 
\end{lemma}
\begin{lemma}[{see e.g. \cite[Prop. 5.16]{lee2011introduction}}]
\label{lemConsistNatural}
Let $[K]$ and $[K']$ be two naturally oriented $n$-simplices in $\R^n$ that share a facet, but have disjoint interiors. Then $[K]$ and $[K']$ are consistently oriented.
\end{lemma}

\begin{remark}
Choosing an orientation of a triangle $T$ in $\R^3$ is equivalent to deciding on a unit normal vector $\vec n$ on $T$: the orientation $[\rm A,B,C]$ corresponds to the vector \begin{equation}
\label{eq:defNormalOrient}
\vec n = \frac{\overrightarrow{AB} \times \overrightarrow{AC}}{\norm{\overrightarrow{AB} \times \overrightarrow{AC}}}\,.
\end{equation}
\end{remark}

\subsection{Triangulations}

An $n$-dimensional {\em triangulation} $\mathcal{M}$ is a finite set of $n$-simplices that we call the {\em elements} of $\mathcal{M}$. The {\em $d$-subsimplices} of $\mathcal{M}$ are defined as
\begin{equation}
\label{eq:defSigmadM}
\nc{\sigma_d(\mathcal{M})} \isdef \bigcup_{K \in \mathcal{M}} \sigma_d(K)\,, \quad \nc{\mathcal{F}(\mathcal{M})} \isdef \sigma_{n-1}(\mathcal{M})\,,\quad \nc{\sigma(\mathcal{M})} \isdef \bigcup_{d \leq n} \sigma_d(\mathcal{M})\,.
\end{equation}
Thus for a triangulation $\mathcal{M}$, $\sigma(\mathcal{M})$ is a {\em pure} abstract simplicial complex \cite[Chap. 7]{moise}. In an abstract triangulation $\mathcal{M}$, two subsimplices $S \in \sigma_d(\mathcal{M})$ and $S' \in \sigma_{d'}(\mathcal{M})$ are {\em incident} if 
$$\abs{d - d'} = 1 \textup{ and } S \subset S' \textup{ or } S' \subset S\,.$$ 
Two $d$-subsimplices $S$ and $S'$ are {\em adjacent} if they are incident to a common $(d-1)$-subsimplex. In this case, we write
\nc{$S\underset{\mathcal{M}}{\longleftrightarrow}S'$}. We say that a triangulation $\mathcal{M}$ is {\em branching} if at least one facet of $\mathcal{M}$ is incident to more than two elements of $\mathcal{M}$, and {\em non-branching} otherwise.

The {\em star} \cite{bryant} of a subsimplex $S \in \sigma(\mathcal{M})$ is the triangulations defined by 
\begin{eqnarray}
\nc{\textup{st}(S,\mathcal{M})} &\isdef& \enstq{K \in \mathcal{M}}{S \subset K}\,, \label{eq:defStarM}
\end{eqnarray}
If $n \geq 1$, the {\em boundary} $\partial \mathcal{M}$ of an $n$-dimensional triangulation $\mathcal{M}$ is the $(n-1)$-dimensional triangulation defined by
\begin{equation}
\label{eq:defBoundM}
\partial \mathcal{M} \isdef \enstq{F \in \mathcal{F}(\mathcal{M})}{F \textup{ is incident to exactly one element of } \mathcal{M}}\,.
\end{equation} 
When $n = 0$, $\partial \mathcal{M}$ is the empty set. One can check that if $\mathcal{M}$ is non-branching, then $\partial \partial \mathcal{M} = \emptyset$. 
The boundary of a non-branching triangulation can be branching. 
\begin{example}
\label{exTriangulations}
Let $\mathcal{V}=  \{A,B,C,D,E\}$ and
$$\mathcal{M} = \{\{A,B,C\},\{A,B,D\}\,,\{A,C,D\}\,,\{B,C,D\}\,,\{C,D,E\}\}\,,$$ then $\partial \mathcal{M} = \{\{C,E\}\,,\{D,E\}\}$, $\partial \partial \mathcal{M} = \{\{C\},\{D\}\}$ and $\partial \partial \partial \mathcal{M} = \emptyset$. The triangulation 
\begin{equation}
\label{counterExNonBranch}
\mathcal{M} = \{\{A,B,C\},\{C,D,E\}\}
\end{equation}
is non-branching, while its boundary is branching. \hfill$\triangle$
\end{example}

An {\em orientation} of a triangulation $\mathcal{M}$ is a choice of an orientation of each of its simplices. An {\em oriented triangulation} is a triangulation equipped with an orientation. If $\mathcal{M}$ is oriented, the oriented simplex corresponding to the element $K \in \mathcal{M}$ is denoted by $[K]_{\mathcal{M}}$. 

The orientation of $\mathcal{M}$ is called {\em compatible} if for any two adjacent elements $K$ and $K'$, $[K]_{\mathcal{M}}$ and $[K']_{\mathcal{M}}$ are consistently oriented. An abstract triangulation which admits a compatible orientation is called {\em orientable}. It is easy to see that an orientable triangulation must be non-branching. Therefore, the triangulation \eqref{counterExNonBranch} of Example \ref{exTriangulations} shows that the boundary of an orientable triangulation may be non-orientable boundary. 

\subsection{Meshes}

An $n$-dimensional triangulation $\mathcal{M}$ is called a {\em mesh} if its vertex set $\mathcal{V}$ is a subset of $\R^m$ with $n\leq m$, and if 
\begin{equation}
\label{condMesh}
\forall (K,K') \in \mathcal{M}\times \mathcal{M}\,, \quad \abs{K} \cap \abs{K'} = \abs{K\cap K'}\,.
\end{equation}
In other words, the intersection of the convex hulls of two elements is either empty or equal to the convex hull of a common subsimplex. 
We write
\begin{equation}
\label{eq:defGeometry}
\nc{\abs{\mathcal{M}}} \isdef \bigcup_{K\in \mathcal{M}} \abs{K}\,,
\end{equation}
and say that $\mathcal{M}$ is {\em regular} if $\abs{\mathcal{M}}$ is an $n$-manifold with or without boundary, meaning that each point $x \in \abs{\mathcal{M}}$ has a neighborhood in $\abs{\mathcal{M}}$ that is homeomorphic to either $\R^n$ or $\R^n_+ \isdef \R^{n-1} \times \R^+$. An example of a non-regular mesh is given by \Cref{example0}, while \Cref{fig:link} shows the the star of a vertex in a regular mesh.

\begin{figure}
\centering
\begin{minipage}[t]{0.45\textwidth}
\centering
\includegraphics[width=0.8\textwidth]{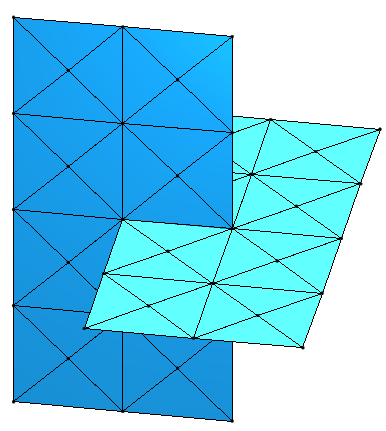}
\caption{A non-regular triangular mesh.}
\label{example0} 
\end{minipage}\hfill
\begin{minipage}[t]{0.45\textwidth}
\centering
\includegraphics[width=\textwidth]{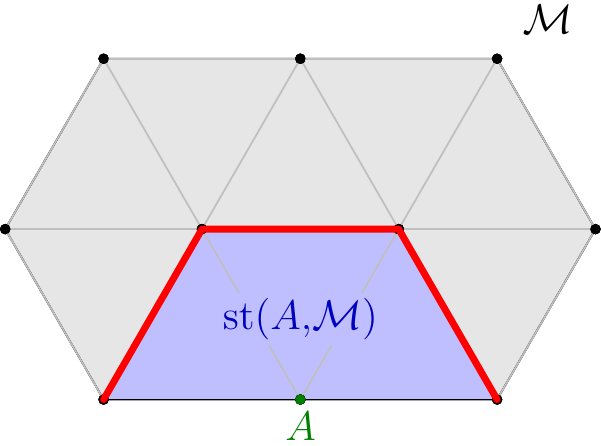}
\caption{A regular triangular mesh. The star of the vertex $A$, is represented in blue.}
\label{fig:link}
\end{minipage}
\end{figure}

%It is a well-known result of piecewise linear topology (see e.g. \cite[Ex. 2.21(i)]{rourke}) that an $n$-dimensional mesh is regular if and only if the link of any vertex is piecewise linearly homeomorphic to either $[0,1]^{n-1}$ or $\partial ([0,1]^{n})$.\footnote{The definition of a piecewise linear homeomorphism can be found e.g. in \cite[p.3]{bryant}.} A simplicial complex satisfying this property is called a triangulation in \cite{boissonnat}, and a combinatorial $n$-manifold in \cite{bryant}. 

Every facet of a regular mesh is incident to at most two elements (see e.g. \cite[Lemma 11.1.2]{boissonnat}). Hence, when $n=3$, this notion of regular mesh corresponds to that of Ciarlet \cite[Chap. 2]{ciarlet}, which is well-established in finite element literature.  In particular, regular meshes are non-branching. Moreover, when $\mathcal{M}$ is regular, then so is its boundary  $\partial \mathcal{M}$, and $\abs{\partial \mathcal{M}}$ is the manifold boundary of $\abs{\mathcal{M}}$ (cf. \cite[p. 4]{bryant}).

We also record the following classical property of regular meshes for later use.
\begin{lemma}
\label{boisLemma}
If $\mathcal{M}$ is a regular mesh and $S \in \sigma(\mathcal{M})$, then the star of $S$ is \emph{face-connected}, in the sense that for any two elements $K$ and $K'$ in $\textup{st}(S,\mathcal{M})$, there are elements $K_1\,,\ldots\,,K_Q$ of $\textup{st}(S,\mathcal{M})$ such that
\begin{itemize}
\item[(i)] $K_1 = K$
\item[(ii)] $K_Q = K'$
\item[(iii)] $\forall i \in \{1\,,\ldots\,,Q-1\}\,,\,\, K_i \underset{\mathcal{M}}{\longleftrightarrow}K_{i+1}\,.$
\end{itemize} 
\end{lemma}
\begin{proof}
Following the comment below Theorem 2.2 in \cite{bryant}, the mesh $\textup{st}(S,\mathcal{M})$ is piecewise linearly homeomorphic to $[0,1]^{n}$. In particular, it is a connected regular mesh, so the result follows from \cite[Lemma 11.1.3]{boissonnat}.
\end{proof}

\subsection{Data structures}
In our Matlab implementation \cite{matlabCode}, meshes over $\mathcal{V} \subset \R^3$ are manipulated via the toolbox {\tt openMsh} from \cite{gypsilab}. The mesh class declaration reads\\

\begin{sampleCode}[plain]{}
\CLASS{\,msh}\\
\mbox{\bfseries properties}\\
\begin{array}{llll}
&{\tt vtx\,;} & & \texttt{\% Nvtx x 3 array of reals coordinates}\\
&{\tt  elt\,;} & &\texttt{\% Nelt x (n+1) array with entries in \{1,...,Nvtx\}}
\end{array}\\
\mbox{\bfseries end}
\end{sampleCode}\\
The points of $\mathcal{V}$ are stored in the array $\tt vtx$. Each line of {\tt vtx} corresponds to a vertex $V \in \mathcal{V}$, the three column giving the $x, y$ and $z$ coordinates. The lines of the array $\tt elt$ are mutually distinct and encode the simplices of $\mathcal{M}$, referring to the vertices by their index in $\tt vtx$. This is essentially the same data structure as in the widely used meshing tool Gmsh \cite{gmsh}. Many alternative choices of data structure exist to represent more general simplicial complexes, see \cite{de2005data} for a review. 

\section{Generalized meshes}
\label{sec:genmeshes}

We now introduce generalized meshes. Roughly speaking, they are defined by a set of elements -- each of which is assigned a realization as an $n$-simplex -- and a specified adjacency between those elements. The adjacency relation must respect the following conditions: (i)  each element has at most one neighbor per facet, and (ii) two elements can only be adjacent if their realizations have a common facet. A more precise definition is given below, and is illustrated by several examples afterwards. We then define generalized subfacets, and briefly describe an algorithm to compute them.

\subsection{Definitions and examples}

\begin{mdframed}
\begin{definition}[Generalized mesh, or ``gen-mesh"]
\label{defGenMesh}
For $n \geq 0$, an $n$-dimensional {\em generalized mesh} $\mathcal{M}^*$ is a quadruple $(\mathcal{V}_{\mathcal{M}^*},\mathbf{K}_{\mathcal{M}^*},\mathcal{K}_{\mathcal{M}^*},\mathcal{G}_{\mathcal{M}^*})$ where
\begin{itemize}
\item[$\bullet$] $\mathcal{V}_{\mathcal{M}^*}$ is a set called the {\em vertex set} of $\mathcal{M}^*$. 
\item[$\bullet$] $\mathbf{K}_{\mathcal{M}^*}$ is a finite set. Its elements are called {\em elements} of $\mathcal{M}^*$; we write $\mathbf{k} \in \mathcal{M}^*$ as short for $\mathbf{k} \in \mathbf{K}_{\mathcal{M}^*}$. 
\item[$\bullet$] $\mathcal{K}_{\mathcal{M}^*}$ is a {\em realization function}, mapping each element $\mathbf{k} \in \mathcal{M}^*$ to a simplex $\mathcal{K}_{\mathcal{M}^*}(\mathbf{k})$ over $\mathcal{V}_{\mathcal{M}^*}$. The simplex $\mathcal{K}_{\mathcal{M}^*}(\mathbf{k})$ 
is called {\em the simplex attached to $\mathbf{k}$}. If $\mathcal{V}_{\mathcal{M}^*} \subset \R^m$ for some $m \geq n$, then the simplices attached to the elements of $\mathcal{M}^*$ must be non-degenerate. The realization function need not be injective: the same simplex may be attached to several distinct elements. 
\item[$\bullet$] $\mathcal{G}_{\mathcal{M}^*}$ is a graph between the \nc{split facets} of $\mathcal{M}^*$, by which we mean the pairs of the form 
\begin{equation}
	\label{eq:defSplitFacet}
	(F,\mathbf{k})\,,\,\,\textup{with}\,\,\, \mathbf{k} \in \mathcal{M}^*  \textup{ and } F \textup{ a facet of } K = \mathcal{K}_{\mathcal{M}^*}(\mathbf{k})\,.
\end{equation}
We write $\nc{\mathbb{F}(\mathcal{M}^*)}$ for the set of split facets. The graph $\mathcal{G}_ {\mathcal{M}^*}$ is called {\em the adjacency graph} of $\mathcal{M}^*$, and is assumed to satisfy the following axioms:
\begin{itemize}
	\item[(i)] the nodes of $\mathcal{G}_{\mathcal{M}^*}$ have degree $0$ or $1$, and
	\item[(ii)] if two split facets $(F,\mathbf{k})$ and $(F',\mathbf{k'})$ are connected by an edge in $\mathcal{G}_{\mathcal{M}^*}$, then 
	$${F = F'} \textup{ and } \mathbf{k} \neq \mathbf{k}'\,.$$
\end{itemize}. 
\end{itemize}
\end{definition}
\end{mdframed}
\begin{remark}
\label{rem:facetUse}
The split facets introduced above correspond exactly to {\em facet uses} in classical non-manifold mesh structures \cite[p. 175]{weiler1986topological}. We shall not require other entity uses. Instead, the important concept is that of generalized subfacets, introduced in \Cref{sec:genfacets}.
\end{remark}
\begin{example}[Domain with a crack] 
\label{example1}
Let $\mathcal{M}^*_1$ be the generalized triangular mesh (see \Cref{crackDomainExample}) defined by
\[\mathcal{V}_{\mathcal{M}^*_1} = \{\textup{A}\,,\ldots\,,\textup{J}\}\,, \quad \mathbf{K}_{{\mathcal{M}^*_1}}= \{1\,,\ldots\,,10\}\,,\] 
with the realization function $\mathcal{K}_{\mathcal{M}^*_1} = \mathcal{K}$ defined by 
\[\mathcal{K}(1) = \textup{ABC}\,, \,\,\mathcal{K}(2) = \textup{BCD}\,, \,\,\mathcal{K}(3) = \textup{BDE}\,, \,\,\mathcal{K}(4) = \textup{BEF}\,,\,\, \mathcal{K}(5) = \textup{BFG} \]
\[\mathcal{K}(6) = \textup{ABG}\,, \,\,\mathcal{K}(7) = \textup{AGH}\,,\,\, \mathcal{K}(8) = \textup{AHI}\,,\,\, \mathcal{K}(9) = \textup{AIJ}\,,\,\,\mathcal{K}(10) = \textup{ACJ}.\]
\begin{figure}[H]
\centering
\includegraphics[width=0.5\linewidth]{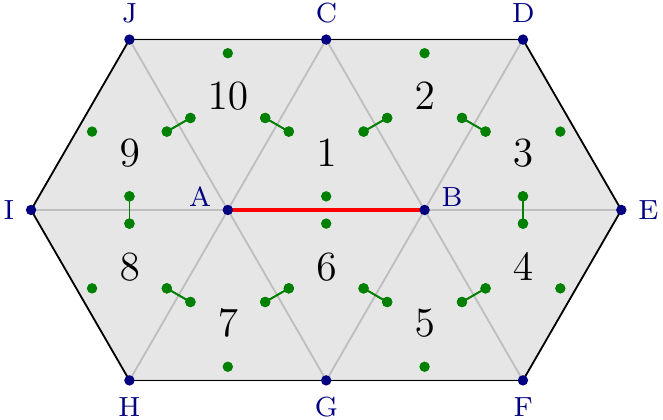}
\caption{Graphical representation of the generalized mesh $\mathcal{M}^*_1$ of Example \ref{example1}. The adjacency graph $\mathcal{G}_{\mathcal{M}^*_1}$ is shown in green. The dot in an element $\mathbf{i}$ and near an edge $E$ is associated to the split facet $(E,\mathbf{i}) \in \mathbb{F}(\mathcal{M}^*_1)$. Even though the elements $1$ and $6$ have the edge AB (highlighted in red) in common, they are not adjacent.}
\label{crackDomainExample}
\end{figure}
\noindent There are $3\times 10 = 30$ split facets (green dots in \Cref{crackDomainExample}), the first few are given by
\[\mathbb{F}(\mathcal{M}^*_1) = \{(\textup{AB},1),(\textup{BC},1),(\textup{AC},1)\,,(\textup{BC},2)\,,(\textup{CD},2),(\textup{BD},2)\,,\ldots\}\,.\]
The adjacency graph is represented by green segments in \Cref{crackDomainExample}. The two split facets (AB,$1$) and (AB,$6$) are not connected: this creates the ``fracture", highlighted in red. \hfill $\triangle$
\end{example}

\label{sec:examples}
\begin{example}[Two-sided segment]
\label{example2}
One can consider the generalized edge mesh $\mathcal{M}^*_2$ (see \Cref{fig:Mstar2}) with 
\[\mathcal{V}_{\mathcal{M}^*_2} = \{\textup{A},\textup{B}\}\,, \quad \mathbf{K}_{\mathcal{M}^*_2} = \{1,2\}\,,\]
and with the realization $\mathcal{K}_{\mathcal{M}^*_2}(1) = \mathcal{K}_{\mathcal{M}^*_2}(2) = \textup{AB}$. There are $4$ split facets:
\[\mathbb{F}(\mathcal{M}^*_2) = \{(\textup{A},1),(\textup{B},1),(\textup{A},2),(\textup{B},2)\}\] and the adjacency graph is represented in green in \Cref{fig:Mstar2}. 
\begin{figure}[H]
\centering
\includegraphics[]{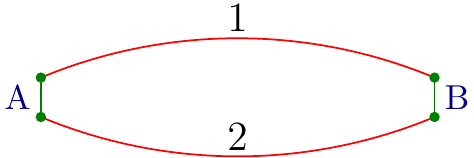}
\caption{Graphical representation of the generalized mesh $\mathcal{M}^*_2$ of Example \ref{example2}, where the two red curves, supposed to be two instances of the segment $\textup{AB}$, shoud be on top of each other. They have been torn apart to visualize the adjacency graph, represented in green.}
\label{fig:Mstar2}
\end{figure}
The two elements of $\mathcal{M}^*_2$ can be thought of as the upper and lower sides of the segment $\textup{AB}$. We shall see that this generalized mesh is one of the components of the boundary of the generalized mesh $\mathcal{M}^*_1$ from the previous example. \hfill $\triangle$
\end{example}

When two split facets $(F,\mathbf{k})$ and $(F',\mathbf{k'})$ of generalized mesh $\mathcal{M}^*$ are connected in the graph $\mathcal{G}_{\mathcal{M}^*}$, we write 
\begin{equation}
\label{eq:notationAdj}
\mathbf{k} \overunderset{F}{\mathcal{M}^*}{\longleftrightarrow} \mathbf{k}'\,,
\end{equation}
and say that $\mathbf{k}$ and $\mathbf{k}'$ are {\em adjacent through $F$}. With this notation, the adjacency graph of $\mathcal{M}^*_1$ from \Cref{example1} can be summarized by
\[1 \overunderset{\textup{BC}}{\mathcal{M}^*_1}{\longleftrightarrow} 2 \overunderset{\textup{BD}}{\mathcal{M}^*_1}{\longleftrightarrow} 3 \overunderset{\textup{BE}}{\mathcal{M}^*_1}{\longleftrightarrow} 4 \overunderset{\textup{BF}}{\mathcal{M}^*_1}{\longleftrightarrow} 5 \overunderset{\textup{BG}}{\mathcal{M}^*_1}{\longleftrightarrow} 6 \overunderset{\textup{AG}}{\mathcal{M}^*_1}{\longleftrightarrow} 7 \overunderset{\textup{AH}}{\mathcal{M}^*_1}{\longleftrightarrow} 8 \overunderset{\textup{AI}}{\mathcal{M}^*_1}{\longleftrightarrow} 9 \overunderset{\textup{AH}}{\mathcal{M}^*_1}{\longleftrightarrow} 10 \overunderset{\textup{AC}}{\mathcal{M}^*_1}{\longleftrightarrow} 1\,,\]
and that of $\mathcal{M}^*_2$ in \Cref{example2} by
\[1 \overunderset{A}{\mathcal{M}^*_2}{\longleftrightarrow} 2 \overunderset{B}{\mathcal{M}^*_2}{\longleftrightarrow} 1\,.\]
If $\mathbf{k} \overunderset{F}{\mathcal{M}^*}{\longleftrightarrow} \mathbf{k}'$ and $S \in \sigma(F)$, then we also say that $\mathbf{k}$ and $\mathbf{k}'$ are adjacent through $S$ and extend the notation $\mathbf{k} \overunderset{S}{\mathcal{M}^*}{\longleftrightarrow} \mathbf{k}'$ to this case. 
For each split facet $(F,\mathbf{k}) \in \mathbb{F}(\mathcal{M}^*)$, by requirement (i), there is at most one element $\mathbf{k}'$ such that $\mathbf{k} \overunderset{F}{\mathcal{M}^*}{\longleftrightarrow} \mathbf{k}'$. If such an an element exist, it is called {\em the neighbor of $\mathbf{k}$} {\em through} $F$. Otherwise, we write $\mathbf{k} \overunderset{F}{\mathcal{M}^*}{\longleftrightarrow}\, \perp$. This way, for each element $\mathbf{k} \in \mathcal{M}^*$, we can define a {\em neighbor function} $$\mathcal{N}_{\mathcal{M}^*}(\mathbf{k},\cdot): \mathcal{F}(K) \to \mathbf{K}_{\mathcal{M}^*} \cup \{\perp\}\,,$$ 
mapping each facet $F$ of the simplex $K = \mathcal{K}_{\mathcal{M}^*}(\mathbf{k})$ to the neighbor of $\mathbf{k}$ through $F$:
\begin{equation}
\label{def:Nmstar}
\forall F \in \mathcal{F}(K)\,, \quad \nc{\mathcal{N}_{\mathcal{M}^*}(\mathbf{k},F)} \isdef \begin{cases}
\mathbf{k'} & \textup{if } \mathbf{k} \overunderset{F}{\mathcal{M}^*}{\longleftrightarrow} \mathbf{k'}\,, \\[1em]
\perp & \textup{if } \mathbf{k} \overunderset{F}{\mathcal{M}^*}{\longleftrightarrow}\, \perp\,. 
\end{cases}
\end{equation}
For example, in the gen-mesh $\mathcal{M}^*_1$ from \Cref{example1}, the neighbors of the element $1$ are given by
\[\mathcal{N}_{\mathcal{M}^*_1}(1,\textup{BC}) = 2\,, \quad \mathcal{N}_{\mathcal{M}^*_1}(1,\textup{AC}) = 10\,, \quad \mathcal{N}_{\mathcal{M}^*_1}(1,\textup{AB}) =\, \perp\,.\]
When the generalized mesh under consideration is sufficiently clear from the context, we drop the subscript $\mathcal{M}^*$ from the notation, i.e. write $\mathcal{V},\mathbf{K}$, $\mathcal{K}$, $\mathcal{G}$, $\overset{F}{\longleftrightarrow}$, instead of $\mathcal{V}_{\mathcal{M}^*},\mathbf{K}_{\mathcal{M}^*}$, $\mathcal{K}_{\mathcal{M}^*}$, $\mathcal{G}_{\mathcal{M}^*}$, $\overunderset{F}{\mathcal{M}^*}{\longleftrightarrow}$, and so on.

\begin{mdframed}
\begin{definition}[Subsimplices of a generalized mesh]
\label{def:subsimplMstar}
Let $\mathcal{M}^*$ be an $n$-dimensional gen-mesh, $\mathbf{k} \in \mathcal{M}^*$ and $K$ the simplex attached to $\mathbf{k}$. The subsimplices and facets of an element $\mathbf{k}$ are defined by
\[\forall d \in \{0,\ldots,n\}\,, \,\, \nc{\sigma_d(\mathbf{k})} \isdef \sigma_d(K)\,, \quad \nc{\sigma(\mathbf{k})} \isdef \sigma(K)\,,  \quad\nc{\mathcal{F}(\mathbf{k})} \isdef \mathcal{F}(K)\,,\]
If $K$ is a non-degenerate simplex of $\R^m$, we also write $\nc{\abs{\mathbf{k}}} \isdef \abs{K}$. 
Similarly, the subsimplices and facets of a generalized mesh $\mathcal{M}^*$ are
\begin{equation}
\label{eq:defSigmadMstar}
\nc{\sigma_d(\mathcal{M}^*)} \isdef \bigcup_{\mathbf{k} \in \mathbf{K}_{\mathcal{M}^*}} \sigma_d(\mathbf{k})\,,\quad \nc{\sigma(\mathcal{M}^*)} \isdef \bigcup_{0 \leq d\leq n} \sigma_d(\mathcal{M}^*)\,, \quad \nc{\mathcal{F}(\mathcal{M}^*)} \isdef \sigma_{n-1}(\mathcal{M}^*)\,.
\end{equation}
\end{definition}
\end{mdframed}
\begin{mdframed}
\begin{definition}[Relabeling of generalized meshes]
\label{defRelabeling}
The gen-mesh $\mathcal{M}_2^*$ is a {\em relabeling} of the gen-mesh $\mathcal{M}^*_1$ if
\begin{itemize}
\item there is a bijection $\varphi:\mathbf{K}_{\mathcal{M}_1^*} \to \mathbf{K}_{\mathcal{M}_2^*}$,
\item the realizations functions $\mathcal{K}_{\mathcal{M}^*_1}$ and $\mathcal{K}_{\mathcal{M}^*_2}$ satisfy $$\mathcal{K}_{\mathcal{M}^*_1} = \mathcal{K}_{\mathcal{M}^*_2} \circ \varphi\,,$$
\item there holds 
$$\mathbf{k} \overunderset{F}{\mathcal{M}^*_1}{\longleftrightarrow}\mathbf{k}' \iff \varphi(\mathbf{k}) \overunderset{F}{\mathcal{M}^*_2}{\longleftrightarrow}\varphi(\mathbf{k}')\,.$$
\end{itemize}
\end{definition}
\end{mdframed}

\noindent Every non-branching triangulation can also be regarded as a generalized mesh, in the following sense:
\begin{mdframed}
\begin{definition}[Generalized mesh respresenting a non-branching triangulation]
\label{GenMeshofProperMesh}
The generalized mesh $\mathcal{M}^*$ {\em representing} the non-branching triangulation $\mathcal{M}$ is defined as follows:
\begin{itemize}
\item[$\bullet$] the vertex set is the same as that of $\mathcal{M}$, i.e. $\mathcal{V}_{\mathcal{M}^*} \isdef \sigma_0(\mathcal{M})$,
\item[$\bullet$] the elements of $\mathcal{M}^*$ are the elements of $\mathcal{M}$, i.e. $\mathbf{K}_{\mathcal{M}^*}\isdef \mathcal{M}$, 
\item[$\bullet$] the realization of $\mathcal{M}^*$ is the identity map, i.e. $\mathcal{K}_{\mathcal{M}^*}(K) = K\,,\,\, \forall K \in \mathcal{M}$,
\item[$\bullet$] the adjacency graph is defined by 
\[K \overunderset{F}{\mathcal{M}^*}{\longleftrightarrow} K' \iff K \overunderset{F}{\mathcal{M}}{\longleftrightarrow} K' \iff F = K\cap K' \,.\]
\end{itemize}
\end{definition}
\end{mdframed}
It is necessary for $\mathcal{M}$ to be non-branching for the last point in the definition above to be consistent with the requirement (i) of \Cref{defGenMesh}. 

Finally, one can define a notion of orientability for generalized meshes as for usual meshes:
\begin{mdframed}
\begin{definition}[Orientable generalized mesh]
\label{defOrientableGenMesh}
An {\em oriented} gen-mesh is a gen-mesh endowed with an {\em orientation}, i.e. a choice of orientation for each of its elements. For an element $\mathbf{k} \in \mathcal{M}^*$, an orientation of $\mathbf{k}$ is an orientation of its realization $\mathcal{K}_{\mathcal{M}^*}(\mathbf{k})$. We denote the corresponding oriented simplex by 	$[\mathbf{k}]_{\mathcal{M}^*}$. 

The orientation of $\mathcal{M}^*$ is called {\em compatible} if, whenever $\mathbf{k} \overset{F}{\longleftrightarrow}\mathbf{k}'$, the simplices $[\mathbf{k}]_{\mathcal{M}^*}$ and $[\mathbf{k'}]_{\mathcal{M}^*}$ are consistently oriented. A generalized mesh that admits a compatible orientation is called {\em orientable}. 
\end{definition}
\end{mdframed}

\subsection{Example of implementation}
\label{sec:dataStruct}
We represent generalized meshes (in Matlab) using the following data structure:\\

\begin{sampleCode}[plain]{}
\CLASS{GeneralizedMesh}\\
\\
\mbox{\bfseries properties}\\
\begin{array}{llll}
&{\tt vtx\,;} && \texttt{\% List of vertices of size Nvtx}\\
&{\tt  elt\,;} & &\texttt{\% Nelt x (n+1) array with entries in \{1,...,Nvtx\}}\\
&{\tt nei\_elt\,;} && \texttt{\% Nelt x (n+1) array with entries in \{0,...,Nelt\}}\\
&{\tt nei\_fct\,;} && \texttt{\% Nelt x (n+1) array with entries in \{0,...,n+1\}}\\
\end{array}\\
\mbox{\bfseries end}
\end{sampleCode}\\
The attribute $\tt elt$ is used as in the $\tt msh$ class, except this time, the rows of $\tt elt$ are not necessarily mutually distinct. An instance of $\tt GeneralizedMesh$ represents a generalized mesh with elements $\{1\,,\ldots\,,N_{elt}\}$. The $i$-th line of $\tt elt$ encodes the simplex attached to the element $i$. The adjacency graph is encoded by the attributes $\tt nei\_elt$, $\tt nei\_fct$ with the following conventions:
\[i \overset{F}{\longleftrightarrow} j \iff \begin{cases}
(\texttt{nei\_elt[$i$,$\alpha$]},\texttt{nei\_fct[$i$,$\alpha$]} )= (j,\beta)\,, &\\
F_\alpha(i) = F_\beta(j) = F\,,&
\end{cases}\]
\[i \overset{F}{\longleftrightarrow} \perp \iff \begin{cases}
(\texttt{nei\_elt[$i$,$\alpha$]},\texttt{nei\_fct[$i$,$\alpha$]} )= (0,0)\,, &\\
F_\alpha(i) = F\,.&
\end{cases}\]
\[\textup{for all } i,j \in \{1\,,\ldots\,,N_{\rm elt}\}\,.\]
Here, for $\alpha \in \{1\,,\ldots\,,n+1\}$, $F_\alpha(i)$ denotes the facet of the simplex attached to the element $i$, obtained by removing the vertex in position $\alpha$ in the $i$-th line of $\tt elt$. 
\begin{example}
\label{ex:dumpMatlab}
For the mesh $\mathcal{M}^*_1$ of \Cref{example1}, this gives the following arrays:
\begin{footnotesize}
\[\texttt{vtx} = \begin{pmatrix}
\textup{A}\\
\textup{B}\\
\textup{C}\\
\textup{D}\\
\textup{E}\\
\textup{F}\\
\textup{G}\\
\textup{H}\\
\textup{I}\\
\textup{J}\\
\end{pmatrix}\,,\quad\texttt{elt} = \begin{pmatrix}
\textup{1} & \textup{2} & \textup{3}\\
\textup{2} & \textup{3} & \textup{4}\\
\textup{2} & \textup{4} & \textup{5}\\
\textup{2} & \textup{5} & \textup{6}\\
\textup{2} & \textup{6} & \textup{7}\\
\textup{2} & \textup{7} & \textup{1}\\
\textup{1} & \textup{7} & \textup{8}\\
\textup{1} & \textup{8} & \textup{9}\\
\textup{1} & \textup{9} & \textup{10}\\
\textup{1} & \textup{10} & \textup{3}\\
\end{pmatrix}\,, \quad \texttt{nei\_elt} = \begin{pmatrix}
2 & 10 & 0 \\
0 & 3 & 1 \\
0 & 4 & 2 \\
0 & 5 & 3 \\
0 & 6 & 4 \\
7 & 0 & 5 \\
0 & 8 & 6 \\
0 & 9 & 7 \\
0 & 10 & 8 \\
0 & 1 & 9 \\
\end{pmatrix}\,,\,\, \texttt{nei\_fct} = \begin{pmatrix}
3 &2 &0 \\
0 &1 & 1\\
0 &3 & 2\\
0 &3 & 2\\
0 &3 & 2\\ 
3 &0 & 2\\ 
0 &3 & 1\\ 
0 &3 & 2\\ 
0 &3 & 2\\ 
0 &2 & 2\\
\end{pmatrix}\]
\end{footnotesize}
with the conventions above.  \hfill$\triangle$
\end{example}

The $d$-subsimplices of a generalized mesh $\mathcal{M}^*$ can be computed by removing duplicate entries in the list of all subsimplices of all elements. This is done by first sorting the list of subsimplices lexicographically (according to the vertex indices), and then removing duplicates in a second, linear pass. Hence, the number of operations required is proportional to $N_1 \log N_1$, where
\begin{equation}
\label{defN1}
N_1 \isdef \binom{n}{d} \#{\mathbf{K}_{\mathcal{M}^*}}\,.
\end{equation}

One of the class constructors for {\tt GeneralizedMesh} is an implementation of \Cref{GenMeshofProperMesh}. Called on an instance of \texttt{msh} representing a non-branching mesh $\mathcal{M}$, it outputs a generalized mesh $\mathcal{M}^*$ representing $\mathcal{M}$. This involves the computation of the adjacency graph of $\mathcal{M}$, which requires a number of operations proportional to $\#{\mathcal{M}}$, the proportionality constant depending polynomially on $n$. 

For more details, we refer to our Matlab implementation \cite{matlabCode}.

\begin{remark}
\label{rem:storage}
One can see that a {\tt GeneralizedMesh} instance representing a non-branching mesh uses up about three times as much memory as the corresponding {\tt msh} instance.  For meshes that do not differ much from a regular mesh, this cost can easily be reduced by only storing those pairs of elements which are not adjacent although they share a facet. Furthermore, the field \texttt{nei\_fct} need not be stored, as it can easily be deduced from \texttt{nei\_elt}. With those simple changes, storing a normal mesh as a {\tt GeneralizedMesh} incurs no extra memory cost compared to an {\tt msh}. 
\end{remark}

\subsection{Generalized subfacets}

\label{sec:genfacets}
The generalized $d$-subfacets defined in this section play an important role in the remainder of this work. 

Given $S \in \sigma(\mathcal{M}^*)$, the {\em generalized star} of $S$ is defined by
\begin{equation}
\label{eq:defGenStar}
\textup{st}(S,\mathcal{M}^*) \isdef \enstq{\mathbf{k} \in \mathcal{M}^*}{S \in \sigma(\mathbf{k})}\,.
\end{equation}
Let \nc{$\mathcal{G}(S)$} be the graph between the elements of $\textup{st}(S,\mathcal{M}^*)$, with an edge between $\mathbf{k}$ and $\mathbf{k}'$  if $\mathbf{k} \overset{S}{\longleftrightarrow} \mathbf{k}'$ (see \Cref{fig:StarGraph}, for an example where $S$ is a vertex).The connected components of $\mathcal{G}(S)$ define a partition $\gamma_1\,,\ldots\,,\gamma_Q$ of $\textup{st}(S,\mathcal{M}^*)$, and we write $$\nc{\mathcal{C}(S)} \isdef \{\gamma_1\,,\ldots\,,\gamma_Q\}\,.$$
\begin{figure}[H]
\centering
\Includegraphics[width=0.35\textwidth]{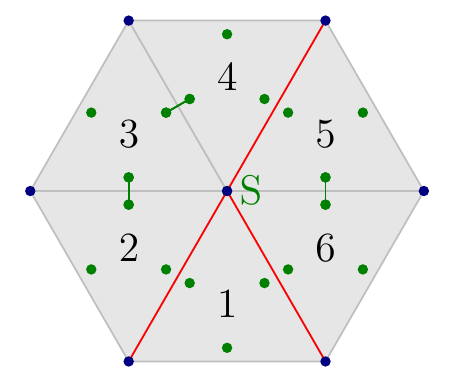} \quad \Includegraphics[width=0.25\textwidth]{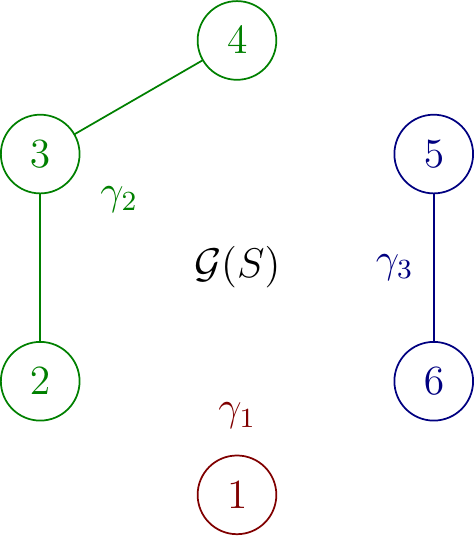}
\caption{Possible configuration of the generalized star of a vertex $S$, with the same conventions as in \Cref{crackDomainExample} (left), and corresponding graph $\mathcal{G}(S)$ (right). In this example, there are three generalized vertices attached to $S$: $(S,\{1\})$, $(S,\{2\,,3\,,4\})$ and $(S,\{5\,,6\})$.}
\label{fig:StarGraph}
\end{figure}
\begin{mdframed}
\begin{definition}[Generalized subfacets]
\label{def:genSubfacet}
For $0 \leq d \leq n$, a {\em generalized $d$-subfacet} is a $d$-simplex $S \in \sigma_d(\mathcal{M}^*)$  labeled by a component $\gamma \in \mathcal{C}(S)$. The set of generalized $d$-subfacets, denoted by $\mathbf{S}_d(\mathcal{M}^*)$, is thus
\begin{equation}
\label{eq:defGenSubfacet}
{\mathbf{S}_d(\mathcal{M}^*)} \isdef \enstq{\mathbf{s} = (S,\gamma)}{S \in \sigma_d(\mathcal{M}^*)\,, \,\, \gamma \in \mathcal{C}(S)}\,.
\end{equation}
We say that a generalized $d$-subfacet $\mathbf{s} = (S,\gamma)$ is {\em attached to} the subsimplex $S$, and we extend the meaning of the realization $\mathcal{K}$ so that  $\nc{\mathcal{K}(\mathbf{s})} \isdef S$.
Also, let us write
\[\nc{\sigma_d(\mathbf{s})} \isdef \sigma_d(S)\,, \quad \nc{\textup{st}(\mathbf{s},\mathcal{M}^*)}\isdef \gamma\,.\]
A generalized $d$-subfacet is called a {\em generalized vertex} and a {\em generalized facet} for $d = 0$ and $d = n-1$, respectively. We again adopt a special notation for the set of generalized facets:
\begin{equation}
\label{eq:defFMstar}
\mathbf{F}(\mathcal{M}^*) \isdef \mathbf{S}_{n-1}(\mathcal{M}^*)\,.
\end{equation}
\end{definition}
\end{mdframed}

For an $n$-dimensional gen-mesh $\mathcal{M}^*$, when $K \in \sigma_n(\mathcal{M}^*)$, the graph $\mathcal{G}(K)$ has no edges. Indeed, the elements in $\textup{st}(K,\mathcal{M}^*)$ can only be adjacent through simplices of dimension $d \leq n-1$, hence, not through $K$. Therefore, 
$$\mathbf{S}_n(\mathcal{M}^*) \simeq \sigma_n(\mathcal{M}^*) = \mathbf{K}_{\mathcal{M}^*}\,,$$ 
and we identify generalized $n$-subfacets with elements. When $\mathcal{M}^*$ represents a regular mesh, then $\mathbf{S}_d(\mathcal{M}^*) \simeq \sigma_d(\mathcal{M}^*)$ for all $0 \leq d \leq n$ because of {Lemma \ref{boisLemma}}.

From the $d$-subsimplices $\sigma_d(\mathcal{M}^*)$, one can compute all the generalized $d$-subfacets of $\mathcal{M}^*$ in a number of operations proportional to $N_2N_3$ where 
$$N_2 = \#{\sigma_d(\mathcal{M}^*)}\,, \quad N_3 = \max_{S \in \sigma_d(\mathcal{M}^*)} \#{\textup{st}(S,\mathcal{M}^*)}\,.$$
This is achieved by a simple graph connected component search, summarized in Algorithms \ref{genfacetsalg} and \ref{auxalg} below:

\begin{algorithm}[H]
\caption{\texttt{GeneralizedSubfacets}$(\mathcal{M}^*,d)$}
\label{genfacetsalg}
\mbox{\bf INPUTS:} Generalized mesh $\mathcal{M}^*$, subfacet dimension $d$. \\
\mbox{\bf RETURNS:} The set of generalized $d$-subfacets of $\mathcal{M}^*$, $\mathbf{S}_d(\mathcal{M}^*)$. \\

$\mathscr{S}_d \GETS \{\}$ \quad {\color{green!50!black}\texttt{\% Initialization}}\\ 
$\mbox{\bf FOR } S \in \sigma_d(\mathcal{M}^*)$ \quad {\color{green!50!black}\texttt{\% Loop over the $d$-subsimplices of $\mathcal{M}^*$}}  \\ 
$\begin{array}{rl}
& l  \GETS \textup{st}(S,\mathcal{M}^*); \quad {\color{green!50!black}\texttt{\% Generalized star of $S$}}\\
&\mbox{\bf WHILE } (l \neq \{\}) \\
&\begin{array}{rl}
& \mbox{Choose } \mathbf{k} \in l; \quad {\color{green!50!black}\texttt{\% Pick element of the star of $S$ not already visited}}\\
& \gamma \GETS \{\};\\
& (\gamma,l) \GETS \texttt{AUX}(\mathcal{M}^*,S,\gamma,l,\mathbf{k}); \quad {\color{green!50!black}\texttt{\% Visit component of $\mathbf{k}$ and store it in $\gamma$.}}\\
& \mathbf{s} \gets (S,\gamma); \quad  {\color{green!50!black}\texttt{\% Create the corresponding new generalized facet}}\\
& \mathscr{S}_d \gets \mathscr{S}_d\cup \{\mathbf{s}\};   {\color{green!50!black}\quad \texttt{\% Append it to the list $\mathscr{S}_d$}}
\end{array}\\
& \mbox{\bf END WHILE}\\
\end{array}$\\

\mbox{\bf RETURN } $\mathscr{S}_d$
\end{algorithm}
\begin{algorithm}[H]
\caption{{\tt AUX}($\mathcal{M}^*,S,\gamma,l,\mathbf{k}$)}
\label{auxalg}
\mbox{\bf INPUTS:} A generalized mesh $\mathcal{M}^*$, a simplex $S \in \sigma(\mathcal{M}^*)$, two disjoint subsets $\gamma$ and $l$ of $\textup{st}(S,\mathcal{M}^*)$, an element $\mathbf{k} \in l$. \\
\mbox{\bf RETURNS} The set $\gamma$ augmented with all elements in the same component as $\mathbf{k}$ in $\mathcal{G}(S)$, and the set $l$ from which those elements have been removed. \\

$\gamma \gets \gamma \cup \{\mathbf{k}\};\quad {\color{green!50!black}\texttt{\% Append element $\mathbf{k}$ to $\gamma$}}\\
l \gets l \setminus \{\mathbf{k}\}; \quad {\color{green!50!black}\texttt{\% Remove element $\mathbf{k}$ from $l$}}\\
\mbox{\bf FOR } F \in \mathcal{F}(\mathbf{k}) \mbox{ such that } S \in \sigma(F)\quad {\color{green!50!black}\texttt{\% Loop over the facets of $\mathbf{k}$ incident to $S$}}\\
\begin{array}{rl}
& \mathbf{k}' \GETS \mathcal{N}(\mathbf{k},F)\quad {\color{green!50!black}\texttt{\% Neighbor of $\mathbf{k}$ through $F$}}\\
& \mbox{\bf IF } \mathbf{k}' \in l\\
& \begin{array}{rl}
& (\gamma,l) \gets \texttt{AUX}(\mathcal{M}^*,S,\gamma,l,\mathbf{k}');\quad {\color{green!50!black}\texttt{\% Recursively visit neighbors of $\mathbf{k'}$ through $S$}}
\end{array}\\
& \mbox{\bf END IF}
\end{array}\\
\mbox{\bf END FOR}\\
\mbox{\bf RETURN } (\gamma,l)$
\end{algorithm}

\begin{mdframed}
\begin{definition}[Incidence and adjacency of generalized subfacets]
\label{defGenInc}
Given $\mathbf{s}$, $\mathbf{s}'\in \mathbf{S}_{d'}(\mathcal{M}^*)$, with $0 \leq d < d' \leq n$, we write $\nc{\mathbf{s} \subset \mathbf{s}'}$ if
\[ \mathcal{K}(\mathbf{s}) \subset \mathcal{K}(\mathbf{s}') \,\,\, \textup{and}\,\,\, \textup{st}(\mathbf{s},\mathcal{M}^*) \supset \textup{st}(\mathbf{s}',\mathcal{M}^*)\,,\]
in whice case we say that $\mathbf{s}$ {\em is contained in} $\mathbf{s}'$.
When $|d - d'| = 1$, we say that $\mathbf{s}$ and $\mathbf{s}'$ are \nc{incident}. Two generalized $d$-subfacets are \nc{adjacent} if they are incident to a common generalized $(d-1)$-subfacet.
\end{definition}  
\end{mdframed}
One can check that a generalized $d$-facet $\mathbf{s}$ is uniquely determined by the set of $(d+1)$ generalized vertices belonging to $\mathbf{s}$, and that for $d = n$, this new notion of adjacency is consistent with the previous one.
\begin{example}
\label{example4}
We consider the generalized mesh $\mathcal{M}^*_3$ represented in \Cref{fig:duplicatedVert} below. It represents a rectangular domain $\Omega$ containing a fracture $\Gamma$ (red edges in the figure). There are $4$ generalized vertices attached to the vertex $O$,
\[\mathbf{v}_1 = (O,\{11,12\})\,,\,\, \mathbf{v}_2 = (O,\{14\})\,,\,\,\mathbf{v}_3 = (O,\{19\})\,,\,\, \textup{ and }\,\, \mathbf{v}_4 = (O,\{21,22\}),\] 
and two generalized edges attached to the edge $\textup{OA}$,
\[\mathbf{e}_1 = (\textup{OA},\{11\}) \quad \mathbf{e}_2 = (\textup{OA},\{14\})\,.\]
In the sense of \Cref{defGenInc}, one has $\mathbf{v}_1 \subset \mathbf{e}_1 \quad\textup{and}\quad\mathbf{v}_2 \subset \mathbf{e}_2$. \hfill $\triangle$
\begin{figure}[H]
\centering
\includegraphics[width=0.70 \textwidth]{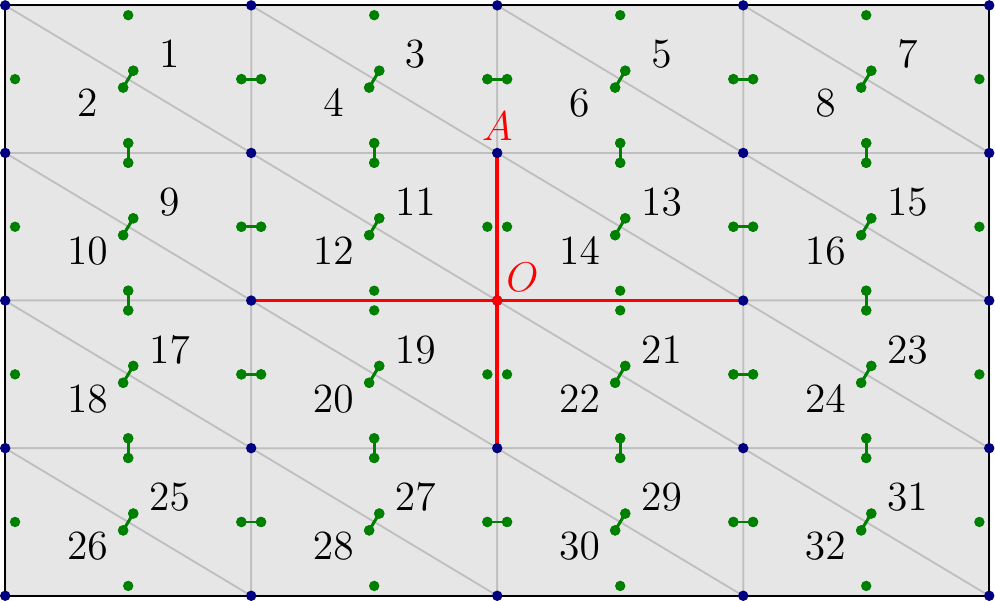}
\caption{A generalized mesh representing a domain with a cross-shaped fracture, with the same conventions as in \Cref{crackDomainExample}.}
\label{fig:duplicatedVert}
\end{figure}
\end{example}

\begin{lemma}
Let $\mathcal{M}^*$ be an $n$-dimensional gen-mesh and suppose $\mathbf{k},\mathbf{k}' \in \mathcal{M}^*$ are adjacent. Then there are $n$ distinct generalized vertices belonging to both $\mathbf{k}$ and $\mathbf{k}'$. The converse is not true, that is, even if two elements share $n$ generalized vertices, they are not necessarily adjacent.
\end{lemma}
\begin{proof}
Let $F \in \mathbf{F}(\mathcal{M}^*)$ be such that $\mathbf{k} \overset{F}{\longleftrightarrow} \mathbf{k}'$.  
Write $F = \{V_1,V_2,\ldots, V_n\}$. For each $k \in \{1\,,\ldots\,,n\}$, since $V_k \in \sigma(F)$ one has
\[\mathbf{k} \overset{V_k}{\longleftrightarrow} \mathbf{k}'\,.\]
Hence, there is a component $\gamma_k \in \mathcal{C}(V_k)$ containing both $\mathbf{k}$ and $\mathbf{k}'$. Set $\mathbf{v}_k \isdef (V_k,\gamma_k) \in \mathbf{S}_0(\mathcal{M}^*)$. Then, the generalized vertices $\mathbf{v}_1\,,\ldots\,,\mathbf{v}_n$ all belong to both $\mathbf{k}$ and $\mathbf{k}'$. 	

For the converse statement, a counter-example is provided by the generalized mesh $\mathcal{M}^*_1$ of \Cref{example1}. The elements $1$ and $6$ share two generalized vertices (attached to $\textup{A}$ and $\textup{B}$, respectively), but they are not adjacent. 
\end{proof}
\begin{remark}
There is a natural concept of a mesh refinement for generalized meshes (implemented in \cite{matlabCode}). In many cases, it seems that given generalized mesh $\mathcal{M}^*$, there exists a refinement $\mathcal{M}^*_{\rm ref}$ of $\mathcal{M}^*$ in which two elements are adjacent {\em if and only if} they share $n$ distinct generalized vertices. When this happens, the mesh $\mathcal{M}^*_{\rm ref}$, and therefore the geometry represented by $\mathcal{M}^*$, can be equivalently described by a non-branching triangulation $\mathcal{M}_{\rm eq}$, defined by
\[\mathcal{V} \isdef \mathbf{S}_0(\mathcal{M}^*_{\rm ref}) \]
\[\mathcal{V}^{n+1} \ni \{\mathbf{v}_1\,,\ldots\,,\mathbf{v}_{n+1}\} \in \mathcal{M}_{\rm eq} \iff \exists \mathbf{k} \in \mathcal{M}^*_{\rm ref} \textup{ such that } \mathbf{v}_1\,,\ldots\,,\mathbf{v}_{n+1} \textup{ belong to } \mathbf{k}\,.\]
One can ask if this is a general fact, that is, whether every generalized mesh $\mathcal{M}^*$ ``reduces" to a triangulation after a suitable refinement. We have not investigated this question yet. 
\end{remark}

\section{Boundary of a generalized mesh}

\label{sec:genbound}
We now define the boundary $\partial^* \mathcal{M}^*$ of a generalized mesh $\mathcal{M}^*$. Recall the definition of the set of split facets $\mathbb{F}(\mathcal{M}^*)$ from Eq.~\eqref{eq:defSplitFacet} and the neighbor function $\mathcal{N}$ from Eq.~\eqref{def:Nmstar}. In what follows, $\mathbf{F}_b(\mathcal{M}^*)$ denotes the set of {\em boundary split facets} of $\mathcal{M}^*$, defined by
\begin{equation}
\label{eq:defFbMstar}
\mathbf{F}_{b}(\mathcal{M}^*) \isdef \enstq{(F,\mathbf{k}) \in \mathbb{F}(\mathcal{M}^*)}{ \mathcal{N}(\mathbf{k},F) = \perp } \,,
\end{equation}
that is, the set of nodes of degree $0$ in the adjacency graph $\mathcal{G}_{\mathcal{M}^*}$.

By definition, $\mathbf{F}_b(\mathcal{M}^*)$ will be the set of elements of $\partial^* \mathcal{M}^*$. To specify the adjacency relation between those elements, we use the following idea, illustrated in \Cref{fig:genBoundAdj}. Let $\mathbf{f}, \mathbf{f}' \in \mathbf{F}_b(\mathcal{M}^*)$, with
$$\mathbf{f} = (F,\mathbf{k})\,, \quad \mathbf{f}' = (F',\mathbf{k}')$$ 
such that $F$ and $F'$ share a common facet $S$. Then we make $\mathbf{f}$ and $\mathbf{f'}$ adjacent through $S$ if $\mathbf{k}$ and $\mathbf{k'}$ can be linked by a chain of adjacent elements circling around $S$. 
\begin{figure}[H]
\centering
\includegraphics[width=0.3\textwidth]{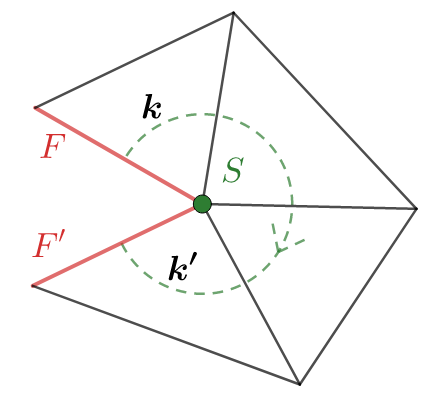}
\caption{Sketch of the idea for the definition of the adjacency relation in the generalized boundary.}
\label{fig:genBoundAdj}
\end{figure}

In the rest of this section, we formalize this idea precisely and prove that it extends the concept of the boundary of a regular mesh. We also show that the boundary of an orientable gen-mesh is orientable. 

The next lemma defines precisely the chain of elements discussed above and establishes its uniqueness.
\begin{lemma}
\label{lemChain}
Let $\mathcal{M}^*$ be an $n$-dimensional generalized mesh, $n\geq 2$, and let $(F,\mathbf{k}) \in \mathbf{F}_b(\mathcal{M}^*)$.
For each $S \in \mathcal{F}(F)$, there exists a unique sequence $\{\mathbf{k}_q\}_{1 \leq q \leq Q} \subset \mathbf{K}_{\mathcal{M}^*}$ of distinct elements, and a unique sequence $\{F_q\}_{0 \leq q \leq Q} \subset \mathcal{F}(\mathcal{M}^*)$ of distinct facets, all containing $S$, such that $\mathbf{k} = \mathbf{k}_{1}$, $F = F_0$, and
\[\perp \,\overset{F_0}{\longleftrightarrow} \mathbf{k}_{1} \overset{F_1}{\longleftrightarrow} \mathbf{k}_{2} \overset{F_2}{\longleftrightarrow}\, \ldots\, \overset{F_{Q-2}}{\longleftrightarrow} \mathbf{k}_{Q-1} \overset{F_{Q-1}}{\longleftrightarrow}  \mathbf{k}_{Q} \overset{F_{Q}}{\longleftrightarrow}\, \perp\,.\]
\end{lemma}
\begin{proof}
Let $\mathbf{k}_{1} \isdef \mathbf{k}$, $F_0 \isdef F$, and let $F_1$ be the unique facet of $\mathbf{k}$, distinct from $F$, containing $S$. If $\mathcal{N}(\mathbf{k},F_1) = \,\perp$, then
\[\perp  \overset{F_0}{\longleftrightarrow}\mathbf{k}_{1} \overset{F_1}{\longleftrightarrow} \,\perp\,,\]
and there is nothing left to prove.

Otherwise, let $\mathbf{k}_{2} \isdef \mathcal{N}(\mathbf{k}_{1},F_1)$, and let the  construction be continued recursively in the following way: $\mathbf{k}_{i}$ and $F_{i-1}$ being defined, let $F_i$ be the facet of $\mathbf{k}_{i}$, distinct from $F_{i-1}$, and containing $S$. If 
\[\mathcal{N}(\mathbf{k}_{i},F_i) = \,\perp\,,\]
then stop with $Q = i$. Otherwise, let
\[\mathbf{k}_{i+1} = \mathcal{N}(\mathbf{k}_{i},F_i)\,.\] 

Let us show by induction that for each $j$ such that $\mathbf{k}_j$ is defined, the elements $\mathbf{k}_1\,,\ldots\,,\mathbf{k}_j$ are all distinct. For $j = 2$, this follows from the requirement (ii) of \Cref{defGenMesh}, i.e. that
\begin{equation}
\label{property}
\mathbf{k} \overset{F}{\longleftrightarrow} \mathbf{k}' \implies \mathbf{k} \neq \mathbf{k}'\,.
\end{equation}
Next, assuming that it is true for some $j \geq 2$, and if $\mathbf{k}_{j+1}$ is defined, let us show by contradiction that $\mathbf{k}_{j+1}\notin \{\mathbf{k}_i\}_{1 \leq i \leq j}$. To this end, we assume that there exists $a < j+1$ such that $\mathbf{k}_a= \mathbf{k}_{j+1}$. The situation is then summarized by the diagram below.
\[\begin{array}{ccccc}
&&\mathbf{k}_{a}  &\overset{F_{a}}{\longleftrightarrow}& \mathbf{k}_{a+1}\\[0.5ex]	
&&\verteq&&
\\[-2.2ex] \mathbf{k}_{j}&\overset{F_{j}}{\longleftrightarrow}&\mathbf{k}_{j+1}&&  
\end{array}\]
Note that $\mathbf{k}_{a+1}\neq \mathbf{k}_{j+1}$,  again by the property \eqref{property} above. Consequently, $a \neq j$, so either $a < j-1$ or $a = j-1$. 

On the one hand, if we assume that $a < j - 1$, then  $\mathbf{k}_{a+1}$ and $\mathbf{k}_{j}$ are two distinct neighbors of $\mathbf{k}_{a}$ through $F_a$ and $F_{j}$, respectively. Hence, $F_j \neq F_a$ so that $F_{a-1} = F_j$ and therefore
\[\mathcal{N}(\mathbf{k}_{a},F_{a-1}) = \mathbf{k}_{j} \neq \,\perp\,.\]
This implies that $a \geq 2$, and $\mathbf{k}_j = \mathbf{k}_{a - 1}$, which is in contradiction with the induction hypothesis.  

On the other hand, assume that $a = j - 1$. We now face the situation represented in the diagram below.
\[\mathbf{k}_{j-1} \overset{F_{j-1}}{\longleftrightarrow} \mathbf{k}_{j} \overset{F_{j}}{\longleftrightarrow} \mathbf{k}_{j+1} = \mathbf{k}_{j-1}\]

By definition of $F_j$, we have $F_j \neq F_{j-1}$, and because of the situation above, both are faces of the simplex attached to $\mathbf{k}_{j-1}$. Hence, we must have $F_{j-2} = F_j$, and therefore
\[\mathcal{N}(\mathbf{k}_{j-1},F_{j-2}) = \mathbf{k}_{j} \neq \,\perp\,.\]
This means that $\mathbf{k}_{j-2}$ is defined and equal to $\mathbf{k}_{j}$, leading to a contradiction also in this case. The existence of $a$ is therefore contradictory, concluding the induction. 

Having established that the sequence of elements generated by this process has no repetition, we conclude that it must be finite, since $\mathcal{M}^*$ only has a finite number of elements. With this, the existence of the chain is proved. 
The fact that the chain is unique follows immediately from the axiom (i) of the adjacency graph in \Cref{defGenMesh}.
\end{proof}

This shows that for each boundary split facet $\mathbf{f} \isdef (F,\mathbf{k}) \in \mathbf{F}_b(\mathcal{M}^*)$, and for each facet $S$ of $F$, there is a ``chain of elements" starting from this boundary split facet and circling around $S$. The opposite end of this chain, $\mathbf{f}' \isdef (F_Q,\mathbf{k}_Q)$ is again a boundary split facet, and we deem it the neighbor of $\mathbf{f}$ through $S$. We write $\mathcal{N}_{b}(\mathbf{f},S) \isdef \mathbf{f}'$.
\begin{corollary}
\label{corBound}
The function $\mathcal{N}_b$ satisfies
\[\forall \mathbf{f} \in \mathbf{F}_{b}(\mathcal{M}^*),\, \forall S \in \mathcal{F}(F)\,, \quad \mathcal{N}_b(\mathbf{f},S) \neq \mathbf{f} \quad \textup{ and }\]
\[\mathcal{N}_b(\mathbf{f},S) = \mathbf{f'}\iff  \mathcal{N}_b(\mathbf{f}',S) = \mathbf{f}\,.\]
\end{corollary}
\begin{mdframed}
\begin{definition}[Generalized boundary]
\label{defGenBound}
Given an $n$-dimensional generalized mesh $\mathcal{M}^*$, with $n \geq 1$, the {\em generalized boundary} (or simply boundary) of $\mathcal{M}^*$ is the $(n-1)$-dimensional generalized mesh $\partial^*\mathcal{M}^*$ with
\begin{itemize}
\item the vertex set $\mathcal{V}_{\partial \mathcal{M}^*} \isdef \mathcal{V}_{\mathcal{M}^*}$,
\item the elements $\mathbf{K}_{\partial^* \mathcal{M}^*} \isdef \mathbf{F}_b(\mathcal{M}^*)$,
\item the realization $\mathcal{K}_{\partial^*\mathcal{M}^*}$, defined by $\mathcal{K}_{\partial^* \mathcal{M}^*}(\mathbf{f}) \isdef F$ when $\mathbf{f} = (F,\mathbf{k})$. 
\item If $n \geq 2$, the adjacency graph $\mathcal{G}_{\partial^* \mathcal{M}^*}$ defined by 
\[\mathbf{f} \overunderset{S}{\partial^* \mathcal{M}^*}{\longleftrightarrow} \mathbf{f}' \iff \mathcal{N}_b(\mathbf{f},S) = \mathbf{f}'\,.\]
\end{itemize}
If $\mathcal{M}^*$ is a $0$-dimensional generalized mesh, we set $\partial^* \mathcal{M}^* \isdef \emptyset$. 
\end{definition}
\end{mdframed}

This definition respects the axioms of generalized meshes by \Cref{corBound}. Moreover, it is clear that for any gen-mesh $\mathcal{M}^*$, the gen-mesh $\partial \partial \mathcal{M}^*$ is empty.

\begin{example}
Consider again the generalized mesh $\mathcal{M}^*_3$ represented in \Cref{fig:duplicatedVert}. Its generalized boundary is represented in \Cref{fig:dM3star} below, where the inner component of the boundary (in red) has been ``inflated" to help visualizing the adjacency structure.  \hfill$\triangle$
\begin{figure}[H]
\centering
\includegraphics[width=0.47\textwidth]{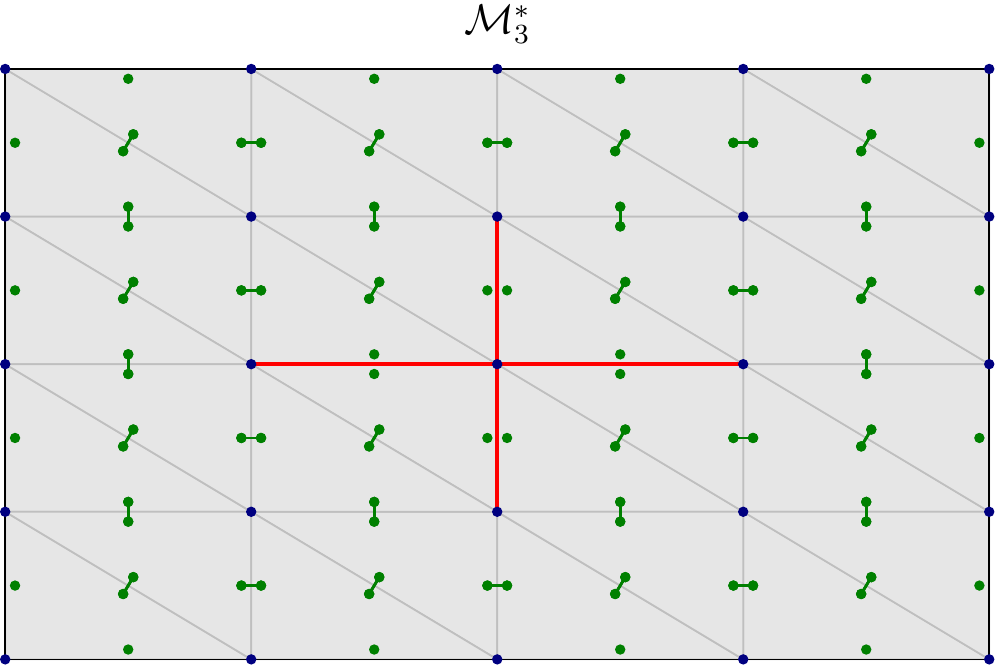} \hfill
\includegraphics[width=0.47\textwidth]{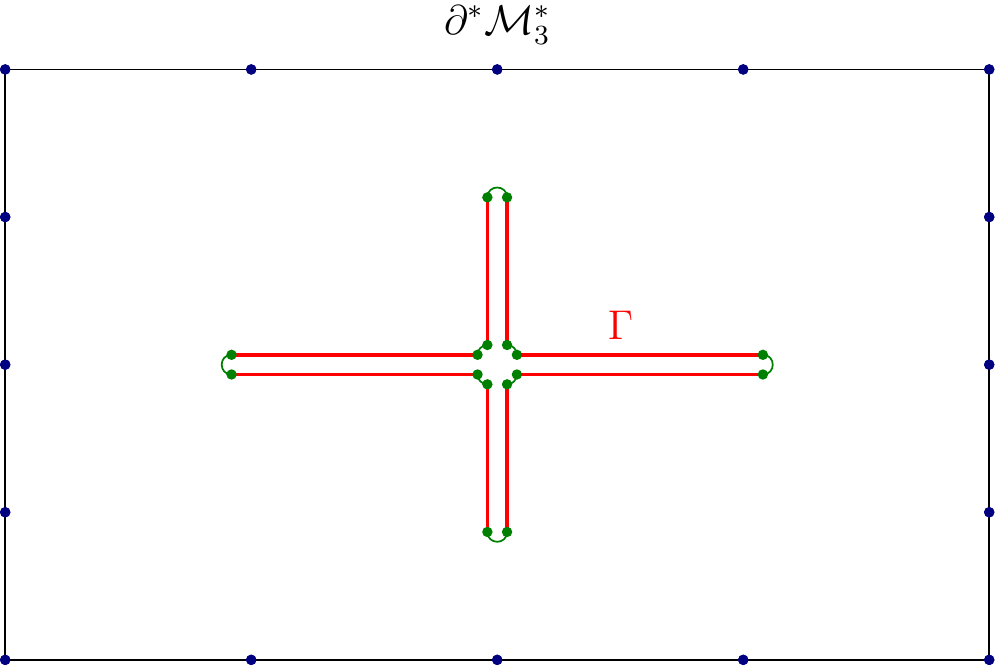}
\caption{Generalized mesh $\mathcal{M}_3^*$ (from \Cref{example4}) and its generalized boundary $\partial^*\mathcal{M}^*_3$. }
\label{fig:dM3star}
\end{figure}
\end{example}
Now, we prove that the previous definition suitably generalizes the notion of boundary for regular meshes.\footnote{We cannot formulate \Cref{lemGenBound} in general for non-branching triangulations, since they can have branching boundaries -- for which there is no corresponding generalized mesh, see the remark below \Cref{GenMeshofProperMesh}.} For the terminology used in the next lemma, recall Definitions \ref{defRelabeling} and \ref{GenMeshofProperMesh}:
\begin{lemma}
\label{lemGenBound}
Let $\mathcal{M}$ be a regular mesh, and let $\mathcal{M}^*$ (resp. $(\partial \mathcal{M})^*$) be the generalized mesh representing $\mathcal{M}$ (resp. $\partial \mathcal{M}$). Then, the meshes 
$\partial^* \mathcal{M}^*$ and $(\partial \mathcal{M})^*$ are equal, up to a relabeling.
\end{lemma}
\begin{proof}
Given $F \in \partial \mathcal{M}$, there exists a unique element $K \in \mathcal{M}$ incident to $F$, and it satisfies
\[K \overunderset{F}{\mathcal{M}}{\longleftrightarrow} \perp\,.\]
Hence, $(F,K) \in \mathbf{F}_b(\mathcal{M}^*)$, and we write $(F,K) =: \varphi(F)$. Since $\mathbf{K}_{(\partial\mathcal{M})^*} = \partial \mathcal{M}$ and $\mathbf{K}_{\partial^* \mathcal{M}^*} = \mathbf{F}_b(\mathcal{M}^*)$, this gives a bijection
\[\varphi: \mathbf{K}_{(\partial\mathcal{M})^*} \to \mathbf{K}_{\partial^* \mathcal{M}^*}\,.\]
Obviously, there holds $\mathcal{K}_{(\partial \mathcal{M})^*} = \mathcal{K}_{\partial^* \mathcal{M}^*} \circ \varphi$. Moreover, given $F,F' \in \mathbf{K}_{(\partial \mathcal{M})^*}$, assume that
$$F \overunderset{S}{(\partial \mathcal{M})^*}{\longleftrightarrow} F'\,.$$ 
Let $K$ and $K'$ be such that $\varphi(F) = (F,K)$ and $\varphi(F') = (F',K')$. Both $K$ and $K'$ are in the star of $S$. Therefore, by \Cref{boisLemma}, there exists $K_1\,,\ldots\,,K_Q$ in the star of $S$, with $K_1 = K$, $K_Q = K'$, such that 
\[K_1 \overunderset{F_1}{\mathcal{M}}{\longleftrightarrow} K_2 \overunderset{F_2}{\mathcal{M}}{\longleftrightarrow} K_2 \overunderset{F_3}{\mathcal{M}}{\longleftrightarrow} \ldots \overunderset{F_{Q-2}}{\mathcal{M}}{\longleftrightarrow} K_{Q-1} \overunderset{F_{Q-1}}{\mathcal{M}}{\longleftrightarrow} K_Q \,,\]
with $F_i \isdef K_i \cap K_{i+1}$, $1 \leq i \leq Q-1$. Furthermore, we have
\[\perp \overunderset{F}{\mathcal{M}}{\longleftrightarrow} K\,, \quad K' \overunderset{F'}{\mathcal{M}}{\longleftrightarrow} \perp\,.\]
By definition of $\mathcal{M}^*$, it follows that
\[\perp \overunderset{F}{\mathcal{M}^*}{\longleftrightarrow}  K \overunderset{F_1}{\mathcal{M}^*}{\longleftrightarrow} K_2 \overunderset{F_3}{\mathcal{M}^*}{\longleftrightarrow} \ldots \overunderset{F_{Q-2}}{\mathcal{M}^*}{\longleftrightarrow} K_{Q-1} \overunderset{F_{Q-1}}{\mathcal{M}^*}{\longleftrightarrow} K' \overunderset{F'}{\mathcal{M}^*}{\longleftrightarrow} \perp\,,\]
that is to say
\[(F',K') = \mathcal{N}_b((F,K),S) \,.\]
We have thus shown
\[F \overunderset{S}{(\partial \mathcal{M})^*}{\longleftrightarrow} F' \implies \varphi(F) \overunderset{S}{\partial^* \mathcal{M}^*}{\longleftrightarrow} \varphi(F')\,.\]
The reverse implication is immediate. Hence $\partial ^*\mathcal{M}^*$ is a relabeling of $(\partial \mathcal{M})^*$.  
\end{proof}
From now on, we drop the star from $\partial^*$. 
\begin{mdframed}
\begin{definition}[Induced orientation on the boundary of a generalized mesh]
\label{defInducedOrientGenMesh}
Let $\mathcal{M}^*$ be an $n$-dimensional generalized mesh, with $n \geq 1$, equipped with some (possibly non-compatible) orientation. We define an orientation of $\partial \mathcal{M}^*$ as follows: for each element $\mathbf{f} = (F,\mathbf{k}) \in \mathbf{F}_b(\mathcal{M}^*)$ of $\mathcal{M}^*$, we choose 
\[[\mathbf{f}]_{\partial\mathcal{M}^*} \isdef [F]_{|[K]}\]
where $K$ is the simplex attached to $\mathbf{k}$, $[K] = [\mathbf{k}]_{\mathcal{M}^*}$ is fixed by the orientation of $\mathcal{M}^*$, and $[F]_{[K]}$ is defined by Eq.~\eqref{inducedOrientFacet}. This orientation of $\partial \mathcal{M}^*$ is called the orientation {\em induced by} $\mathcal{M}^*$. 
\end{definition}
\end{mdframed}
\begin{lemma}
\label{lemInducedOrientBound}
If the orientation of $\mathcal{M}^*$ is compatible, then the induced orientation of $\partial \mathcal{M}^*$ is compatible.
\end{lemma}
\begin{proof}
Let $\mathcal{M}^*$ be a generalized mesh equipped with a compatible orientation. 
Consider two elements $\mathbf{f} = (F,\mathbf{k})$ and $\mathbf{f}' = (F',\mathbf{k'})$ such that 
\[\mathbf{f} \overunderset{S}{\partial \mathcal{M}^*}{\longleftrightarrow} \mathbf{f}'\,.\]
We may introduce the chain
\begin{equation}
\perp \overunderset{F_0}{\mathcal{M}^*}{\longleftrightarrow} \mathbf{k}_1  \overunderset{F_1}{\mathcal{M}^*}{\longleftrightarrow} \mathbf{k}_2 \overunderset{F_2}{\mathcal{M}^*}{\longleftrightarrow} \,\ldots\, \overunderset{F_{Q-2}}{\mathcal{M}^*}{\longleftrightarrow}\mathbf{k}_{q-1} \overunderset{F_{Q-1}}{\mathcal{M}^*}{\longleftrightarrow} \mathbf{k}_Q\overunderset{F_Q}{\mathcal{M}^*}{\longleftrightarrow}\,\perp\,,
\end{equation}
where $F_0 = F$, $\mathbf{k}_1 = \mathbf{k}$, $F_Q = F'$ and $\mathbf{k}_Q = \mathbf{k}'$. We define 
\[\begin{array}{lcll}
\,[F_q]_{\textup{left}} &\isdef& \textup{orientation of $F_q$ induced by } [\mathbf{k}_q]_{\mathcal{M}^*} & q \in \{1\,,\ldots\,,Q\}\,,\\
\,[F_q]_{\textup{right}} &\isdef& \textup{orientation of $F_q$ induced by } [\mathbf{k}_{q+1}]_{\mathcal{M}^*} & q \in \{0\,,\ldots\,,Q-1\}\,.
\end{array} \]
On the one hand, by \Cref{lemConsist}, it holds that
\[\forall q \in \{1\,,\ldots\,,Q\}\,, \quad [F_{q-1}]_{\textup{right}} \textup{ and } [F_{q}]_{\textup{left}} \textup{ induce opposite orientations on $S$}\,,\]

On the other hand, by the property that $[\mathbf{k}_q]_{\mathcal{M}^*}$ and $[\mathbf{k}_{q+1}]_{\mathcal{M}^*}$ are consistently oriented (since the orientation of $\mathcal{M}^*$ is compatible), it holds that 
\[\forall q \in \{1\,,\ldots\,,Q-1\}\,, \quad [F_{q}]_{\textup{right}} = -[F_{q}]_{\textup{left}}\,,\]
Consequently, $[F_{q}]_{\textup{left}}$ and $[F_q]_{\textup{right}}$ also induce opposite orientations on $S$. 

It results from those two facts that $[F_0]_{\textup{right}}$ and $[F_Q]_{\textup{left}}$ induce opposite orientations on $S$, so they are consistently oriented. The conclusion of the lemma follows once we notice that $[F_0]_{\textup{left}}$ and $[F_Q]_{\textup{left}}$ are nothing else than the orientations of $F$ and $F'$ fixed by the orientation induced by $\mathcal{M}^*$ on $\partial \mathcal{M}^*$ according to \Cref{defInducedOrientGenMesh}. 
\end{proof}

\section{Fractured meshes and virtual inflation}
\label{sec:fracturedMeshes}
In this section, we focus on one particular kind of generalized mesh: those obtained, starting from a regular mesh $\mathcal{M}_\Omega$ of a domain $\Omega$, by ``marking" a subset $\mathcal{M}_\Gamma$ (called the fracture) of the facets of $\mathcal{M}_\Omega$, and, for each marked facet $F$, dropping the adjacency between the two elements of $\mathcal{M}_\Omega$ incident to $F$. The resulting mesh $\mathcal{M}^*_{\Omega \setminus \Gamma}$ is called a {\em fractured mesh} of $\Omega \setminus \Gamma$ (see below for a more precise definition). Particular examples of such meshes have already been encountered in Example~\ref{example1} and Example~\ref{example4}. In what follows, we write $\Gamma \isdef \abs{\mathcal{M}_\Gamma}$; our main concern is when $\Gamma$ is not a manifold. 

Using the boundary of $\mathcal{M}^*_{\Omega \setminus \Gamma}$, in the sense of the previous section, one can naturally associate to $\mathcal{M}_\Gamma$ a generalized mesh $\mathcal{M}^*_\Gamma(\Omega)$, which is a ``two-sided" version of $\mathcal{M}_\Gamma$. In the particular case where $\mathcal{M}_\Gamma$ is a manifold, $\mathcal{M}^*_\Gamma$ corresponds to the orientation covering of $\mathcal{M}_\Gamma$ \cite{hatcher2002algebraic}. 

The gen-mesh $\mathcal{M}^*_\Gamma(\Omega)$ is perfectly suited for the implementation of a boundary element method on the fracture. The main reason for this will become apparent in the next section. There, we show that the generalized $d$-facets of $\mathcal{M}^*_\Gamma$ provide a convenient representation of the space of (multi-valued) restrictions to $\Gamma$ of discrete differential forms in a neighborhood of $\Gamma$ which are allowed to jump across $\Gamma$.

Requiring a mesh of the exterior of the fracture would squander a crucial advantage of boundary element methods. Fortunately, it turns out that the gen-mesh $\mathcal{M}^*_\Gamma(\Omega)$ is independent of $\Omega$, and there exists an efficient intrinsic algorithm to construct it. The main idea was already outlined in \cite{quotientBem}. The purpose of this section is to describe this algorithm formally and prove that it is correct, i.e. that it returns the same generalized mesh as $\mathcal{M}^*_\Gamma(\Omega)$, up to a relabeling.

\subsection{Fractured meshes}
\label{sec:defFracturedMesh}
Consider a regular mesh $\mathcal{M}_\Omega$ of dimension $n \geq 2$ and let $\Omega \isdef \abs{\mathcal{M}}$. Let $\mathcal{M}_\Gamma \subset \mathcal{F}(\mathcal{M}_\Omega)$ be a (not necessarily regular) mesh of dimension $n-1$, called the fracture, and let ${\Gamma \isdef \abs{\mathcal{M}_\Gamma}}$. For example, if $n = 3$, $\mathcal{M}_\Gamma$ may be the mesh represented in \Cref{example0}.

\begin{mdframed}
\begin{definition}[Fractured mesh]
\label{def:fracturedMesh}
Given $\mathcal{M}_\Omega$ and $\mathcal{M}_\Gamma$ fulfilling the conditions above, the {\em fractured mesh} $\mathcal{M}_{\Omega\setminus\Gamma}^*$ is the generalized mesh with
\begin{itemize}
\item the vertex set $\mathcal{V}_{\mathcal{M}^*_{\Omega \setminus \Gamma}} \isdef \sigma_0(\mathcal{M}_\Omega)$,
\item the elements $\mathbf{K}_{\mathcal{M}_{\Omega\setminus\Gamma}^*} \isdef \mathcal{M}_{\Omega}$,
\item the identity realization $\mathcal{K}_{\mathcal{M}^*_{\Omega \setminus \Gamma}}$ i.e. $\forall K \in \mathcal{M}_\Omega\,,\,\,\mathcal{K}_{\mathcal{M}^*_{\Omega \setminus \Gamma}}(K) = K$,
\item the adjacency graph defined by 
\[K \overunderset{F}{\mathcal{M}^*_{\Omega \setminus \Gamma}}{\longleftrightarrow} K' \iff \left( K \underset{\mathcal{M}_\Omega}{\longleftrightarrow} K' \,\,\textup{ and } \,\, F = K \cap K' \notin \mathcal{M}_\Gamma\right)\,.\]
In words, two elements of $\mathcal{M}^*_{\Omega \setminus \Gamma}$ are adjacent if they share a facet $F$ which is not in the fracture.
\end{itemize}
\end{definition}
\end{mdframed}

\subsection{Extrinsic virtual inflation}

We introduce the following definition, which is illustrated in \Cref{fig:procInflExt}.
\label{inflByExt}
\begin{mdframed}
\begin{definition}[Extrisinc inflation]
\label{def:inflByExt}
Given an $n$-dimensional regular mesh $\mathcal{M}_\Omega$ and a mesh \[\mathcal{M}_\Gamma \subset \mathcal{F}(\mathcal{M}_\Omega) \setminus \partial \mathcal{M}_\Omega\,,\]
the {\em extrinsic inflation} of $\mathcal{M}_\Gamma$ via $\mathcal{M}_\Omega$ is the $(n-1)$-dimensional generalized mesh $\mathcal{M}^*_{\Gamma}(\Omega)$ with
\begin{itemize}
\item the vertex set $\mathcal{V}_{\mathcal{M}^*_\Gamma(\Omega)} \isdef \sigma_0(\mathcal{M}_\Gamma)$, 
\item the elements $\mathbf{K}_{\mathcal{M}^*_\Gamma(\Omega)} \isdef \enstq{(F,K)\in \mathcal{M}_\Gamma \times \mathcal{M}_\Omega}{F \in \mathcal{F}(K)}$,
\item the realization $\mathcal{K}_{\mathcal{M}^*_{\Gamma}(\Omega)}$ defined by
\[\mathcal{K}_{\mathcal{M}^*_\Gamma(\Omega)}((F,K)) \isdef F\]
\item the adjacency graph defined by
$$\forall \mathbf{f},\mathbf{f}' \in \mathbf{K}_{\mathcal{M}^*_{\Gamma}(\Omega)}\,,\quad {\mathbf{f}} \overunderset{S}{\mathcal{M}_{\Gamma}^*(\Omega)}{\longleftrightarrow} {\mathbf{f}}' \iff \mathbf{f} \overunderset{S}{\partial \mathcal{M}_{\Omega \setminus \Gamma}^*}{\longleftrightarrow} \mathbf{f}'\,,$$
where $\mathcal{M}^*_{\Omega \setminus \Gamma}$ is the fractured mesh defined in \Cref{def:fracturedMesh}.
\end{itemize}
\end{definition}
\end{mdframed} 
The gen-mesh $\mathcal{M}^*_\Gamma(\Omega)$ is essentially a ``submesh" of $\partial \mathcal{M}^*_{\Omega \setminus \Gamma}$; it is obtained by discarding the component of $\mathcal{M}^*_{\Omega \setminus \Gamma}$ corresponding to $\partial \Omega$. The assumption that $\mathcal{M}_\Gamma \cap \partial \Omega = \emptyset$ ensures that this separation can be done properly, i.e. that the definition above respects the axioms of \Cref{defGenMesh}. For example, if $\mathcal{M}^*_{\Omega \setminus \Gamma}$ is equal to the generalized mesh $\mathcal{M}^*_1$ of Example \ref{example1}, then $\mathcal{M}^*_{\Gamma}(\Omega)$ is equal to the generalized mesh $\mathcal{M}^*_2$ of Example \ref{example2}. The idea is also represented schematically in \Cref{fig:procInflExt}. Compared to $\mathcal{M}_\Gamma$, the gen-mesh $\mathcal{M}^*_\Gamma(\Omega)$ has twice as many elements. The vertices at which $\mathcal{M}_\Gamma$ is locally a manifold give rise to two distinct generalized vertices of $\mathcal{M}^*_\Gamma(\Omega)$, one for each side of $\mathcal{M}_\Gamma$. 

\begin{figure}
\centering
\includegraphics[width=0.25\textwidth]{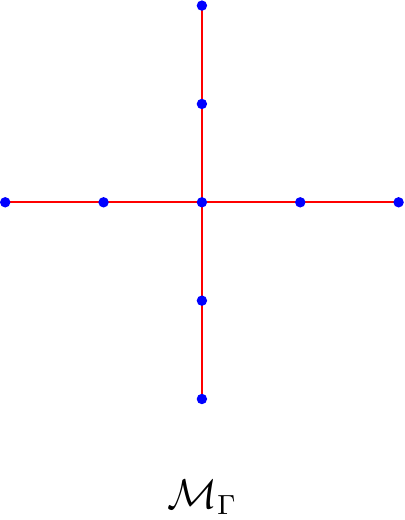} \quad
\includegraphics[width=0.25\textwidth]{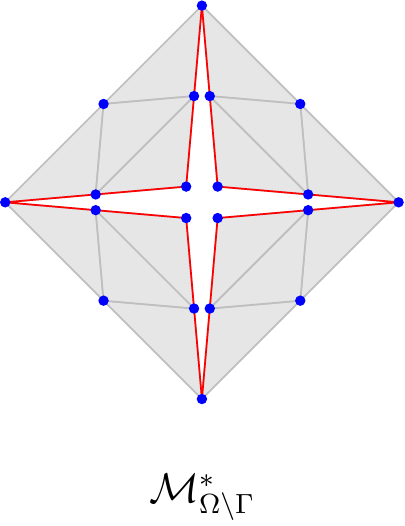} \quad
\includegraphics[width=0.25\textwidth]{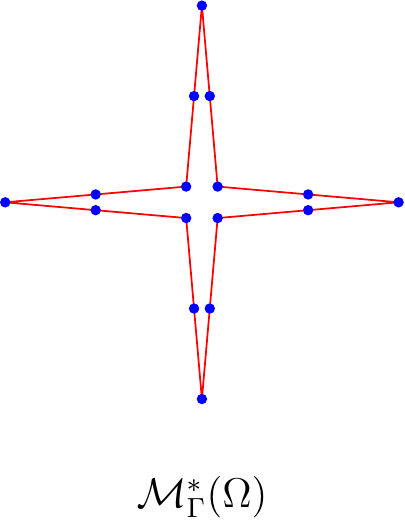}
\caption{Schematic representation of the extrinsic inflation procedure. Starting from a possibly non-regular mesh $\mathcal{M}_\Gamma$ (left panel), form the fractured mesh $\mathcal{M}^*_{\Omega \setminus \Gamma}$ using the exterior regular mesh $\mathcal{M}_\Omega$ (middle), compute its generalized boundary and return the ``submesh" obtained by only keeping components corresponding to $\mathcal{M}_\Gamma$ (right). The blue dots on the left (resp. right) figure represent the vertices (resp. generalized vertices) of $\mathcal{M}_\Gamma$ (resp. of $\mathcal{M}^*_{\Gamma}(\Omega)$). As in \Cref{fig:dM3star}, a gap has been introduced in the center of the cross shape only for visualization purposes. In reality, the four central blue dots share a common location. In the middle figure, only a subset of the elements of $\mathcal{M}_\Omega$ is represented (gray triangles), the exterior boundary $\partial \Omega$ is not displayed. }
\label{fig:procInflExt}
\end{figure}
\Cref{def:inflByExt} suggests a simple, purely combinatorial procedure to compute $\mathcal{M}^*_\Gamma(\Omega)$, which exploits the external mesh $\mathcal{M}_\Omega$. As mentioned above, there is an alternative, intrinsic algorithm. To present it, we need to review some properties of oriented angles in $\R^3$.

\subsection{Oriented angles in $\R^3$}

Let $T_1$ and $T_2$ be two triangles in $\R^3$, sharing an edge. Let us denote their vertices by $\{A,B,C\}$ and $\{B,C,D\}$, respectively. We define the {\em geometric angle} $\Theta(T_1,T_2) \in [0,\pi)$ by 
\begin{equation}
\label{eq:defGeomAngle}
\cos \Theta(T_1,T_2) = \frac{\overrightarrow{OA'} \cdot \overrightarrow{OD'}}{\norm{\overrightarrow{OA'}} \norm{\overrightarrow{OD'}}}\,,
\end{equation}
$C'$ (resp. $D'$) is the orthogonal projections of $C$ (resp. $D$) on the plane perpendicular to $\overrightarrow{BC}$, through the origin $O$, i.e. 
\[\overrightarrow{OA'} = \overrightarrow{OA} -( \overrightarrow{OA}\cdot{\vec u} )\vec u\,, \quad \overrightarrow{OD'} = \overrightarrow{OD} - (\overrightarrow{OD} \cdot \vec u) \vec u\,, \quad \vec u \isdef \frac{\overrightarrow{BC}}{\norm{\overrightarrow{BC}}}\,.\]
Given an orientation $[T_1]$ of $T_1$, we define the {\em oriented angle} $\angle([T_1],T_2) \in (0,2\pi]$ by
\begin{equation}
\label{eq:defOrientedAngle}
\angle([T_1],T_2) = \begin{cases}
\Theta(T_1,T_2) & \textup{ if } \overrightarrow{AD} \cdot \vec n_{T_1} > 0\,, \\
2\pi - \Theta(T_1,T_2)  & \textup{otherwise}\,,
\end{cases}
\end{equation}
where $\vec n_{T_1}$ is the unit normal vector to $T_1$ fixed by its orientation according to Eq.~\eqref{eq:defNormalOrient} (see \Cref{fig:consistOrient}). One has 
\begin{equation}
\label{rulesOrient}
\angle([T_1],T_2) = 2\pi - \angle(-[T_1],T_2)\,.
\end{equation}
In addition, every oriented triangle $[T]$ satisfies $\angle([T],T) = 2\pi$. 
\begin{figure}
\centering
\includegraphics[width=6cm]{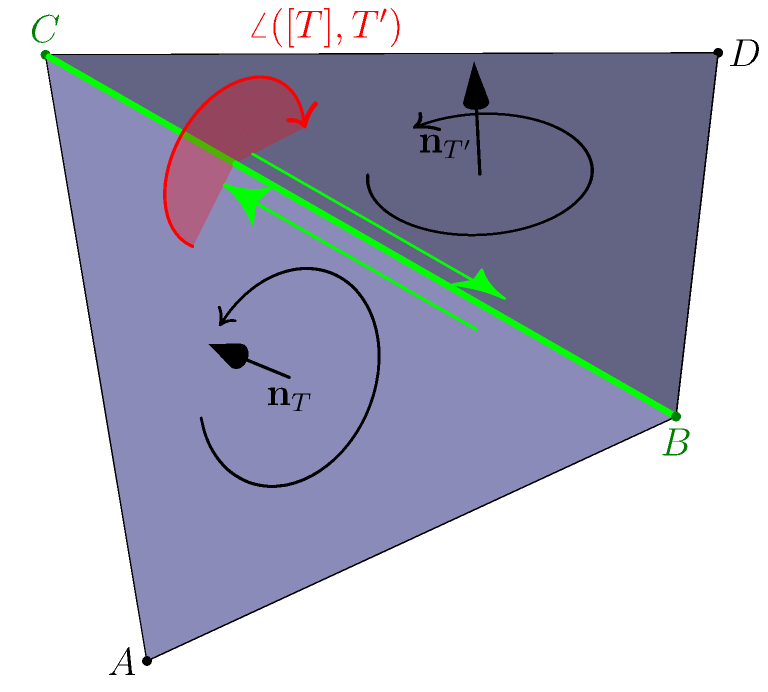}
\caption{The triangles $[T] = [A,B,C]$ and $[T'] = [B,D,C]$ are consistently oriented. The oriented angle $\angle([T],T')$ is drawn in red.}
\label{fig:consistOrient}
\end{figure}

Let $E_1, E_2$ be two edges in $\R^2$ sharing exactly one vertex. Let $[E_1] = [A,B]$ be an orientation of $E_1$ and let $C$ be the vertex of $E_2$ not shared by $E_1$. We define $\angle([E_1],E_2)$ as the counter-clockwise measure in $(0,2\pi)$ of the angle from $\overrightarrow{AB}$ to $\overrightarrow{AC}$ around $A$. We furthermore define $\angle([E_1],E_1) \isdef 2\pi$. With this definition, Eq. \eqref{rulesOrient} also holds for oriented angles between edges in $\R^2$. 
\begin{lemma}
\label{lemAngles}
If $[T_1]$ and $[T_2]$ are consistently oriented triangles in $\R^3$, then
\[\angle([T_1],T_2) = \angle([T_2],T_1)\,.\] 
The same result holds for consistently oriented edges in $\R^2$.
\end{lemma}
\begin{proof} 
Using the rule \eqref{rulesOrient} above, we can assume without loss of generality that $T_1 = [A,B,C]$ and $T_2 = [C,B,D]$. The corresponding normal vector are given by 
\[\vec n_{1} = \frac{\overrightarrow{AB} \times \overrightarrow{AC} }{\norm{\overrightarrow{AB} \times \overrightarrow{AC} }}\,, \quad \vec n_{2} = \frac{\overrightarrow{BD} \times \overrightarrow{BC} }{\norm{\overrightarrow{BD} \times \overrightarrow{BC} }}\,.\] 
By the properties of cross product, one has
\[\overrightarrow{AD} \cdot \left(\overrightarrow{AB} \times \overrightarrow{AC} \right) = \overrightarrow{DA} \cdot \left(\overrightarrow{BD} \times \overrightarrow{BC} \right)\,.\]
We deduce that $\overrightarrow{AD} \cdot \vec n_{1}$ and $\overrightarrow{DA} \cdot \vec n_{2}$ have the same sign, and thus $\angle([T_1],T_2) = \angle([T_2],T_1)$. 
\end{proof}
We conclude our discussion of oriented angles by stating the following elementary property. The proof is omitted for the sake of conciseness.
\begin{lemma}
\label{corObvious}
Let $\mathcal{M}$ be a regular tetrahedral mesh, and let $[K]$ be a naturally oriented element of $\mathcal{M}$. Suppose that $T_1$, $T_2$ are distinct triangular faces of $K$, sharing an edge $E$, and write $[T_1]$ for the orientation of $T_1$ induced by $[K]$. Then it holds that
\[\forall T' \in \mathcal{F}(\mathcal{M})\textup{ s.t. } E \in \mathcal{F}(T')\,, \quad \angle([T_1],T') \leq \angle([T_1],T_2) \implies T' = T_2\,.\] 
In words, $T_2$ is the unique minimizer of $\angle([T_1],\cdot)$ among triangles $T' \in \mathcal{F}(\mathcal{M})$ incident to $E$.
\end{lemma}

%\begin{lemma}
%	\label{lemObvious}
%	Let $[K]$ be a naturally oriented tetrahedron in $\R^3$, with vertices $A$, $B$, $C$ and $D$. Let $E$ be a point in $\R^3$, and denote by $T_1$, $T_2$ and $T_3$ the triangles with vertices $\{A,B,C\}$, $\{B,C,D\}$ and $\{B,C,E\}$, respectively (see Figure \ref{fig:ABCDE}). Let $[T_1]$ be the orientation of $T_1$ induced by $[K]$. Assume that
%	\[\angle([T_1],T_3) < \angle([T_1],T_2)\]
%	Then $T_2$ intersects the interior of $K$, and thus
%	$$\abs{T_2} \cap \abs{K} \neq \abs{T_2 \cap K}\,.$$
%\end{lemma}
%\begin{figure}[H]
%	\centering
%	\includegraphics[width=0.4\textwidth]{geometricLemma}
%	\caption{Configuration of Lemma \ref{lemObvious}.}
%	\label{fig:ABCDE}
%\end{figure}

%\begin{proof}
%	The assumptions imply $\angle(T_1,T_2) = \Theta(T_1,T_2)$ and $\angle(T_2,T_1) = \Theta(T_2,T_1)$, hence
%	$$\angle(T_1,T_2) = \angle(T_2,T_1)\,.$$ 
%	Since $\Theta(T_1,T_2) < \pi$, it follows that
%	\[\angle(\overline{T}_1,T_2) - \angle(T_2,\overline{T}_1) = 2\pi - 2\angle(T_1,T_2) > 0\,.\]
%	By \Cref{lemAngles}, $\overline{T_1}$ and $T_2$ are therefore not consistently oriented, so $T_1$ and $T_2$ are. 
%\end{proof}

\subsection{Intrinsic mesh inflation}
\label{sec:intrinsic}
In what follows, we consider an $(n-1)$-dimensional mesh $\mathcal{M}_\Gamma$ with vertices in $\R^{n}$, where $n = 2$ or $3$. Let 
\begin{equation}
\label{eq:defOrientFacetsDouble}
\mathscr{F}_\Gamma \isdef \{\textup{oriented simplex of the form $[F]$ with $F \in \mathcal{M}_\Gamma$} \}\,.
\end{equation}
Every element $F$ of $\mathcal{M}_\Gamma$ appears twice in $\mathscr{F}_\Gamma$, once with each of the two opposite orientations of $F$.

Given $[F] \in \mathscr{F}_\Gamma$, and $S \in \mathcal{F}(F)$, let $F'$ be the minimizer, among elements $F'\in\mathcal{M}_\Gamma$ incident to $S$ (including $F$ itself), of the quantity $\angle([F],F')$. Then define 
$$\widetilde{\mathcal{N}}_b([F],S) \isdef [F']$$ 
where $[F']$ is the orientation of $F'$ consistent with $[F]$. 
\begin{mdframed}
\begin{definition}[Intrinsic virtual inflation]
\label{defMeshInflGeo}
The {\em intrinsic virtual inflation} of $\mathcal{M}_\Gamma$ is the generalized mesh $\mathcal{M}^*_{\Gamma}$ defined by
\begin{itemize}
\item the vertex set $\mathcal{V}_{\mathcal{M}^*_\Gamma} \isdef \sigma_0(\mathcal{M}_\Gamma)$ 
\item the elements $\mathbf{K}_{\mathcal{M}^*_{\Gamma}} \isdef \mathscr{F}_\Gamma$, 
\item the realization $\mathcal{K}_{\mathcal{M}^*_{\Gamma}}$ defined by 
\[\forall [F] \in \mathscr{F}_\Gamma\,,\quad\mathcal{K}_{\mathcal{M}^*_{\Gamma}}([F]) \isdef F\,,\]
\item the adjacency graph defined by
$$[F] \overunderset{S}{\mathcal{M}^*_{\Gamma}}{\longleftrightarrow} [F'] \iff [F'] = \widetilde{\mathcal{N}}_b([F],S)$$
\end{itemize}
\end{definition}
\end{mdframed}
By \Cref{lemAngles}, and since an oriented triangle is never consistently oriented with itself, this is a well-defined generalized mesh. 
\begin{remark}
\Cref{defMeshInflGeo} can be converted straightforwardly into an algorithm. On the input of a (standard) mesh $\mathcal{M}_\Gamma$, it returns an instance of generalized mesh representing $\mathcal{M}^*_\Gamma$. We refer to our implementation \cite{matlabCode} for details. 
\end{remark} 

\begin{theorem}
\label{theoremInflation}
Let  $\mathcal{M}_\Omega$ be a regular $n$-dimensional mesh in $\R^n$, with $n = 2$ or $3$, and let \[\mathcal{M}_\Gamma \subset \mathcal{F}(\mathcal{M}_\Omega) \setminus \partial \mathcal{M}_\Omega\,.\]
Let $\mathcal{M}^*_{\Gamma}(\Omega)$ be the extrinsic inflation of $\mathcal{M}_\Gamma$ via $\mathcal{M}_\Omega$ and $\mathcal{M}^*_\Gamma$ the intrinsic inflation of $\mathcal{M}_\Gamma$, cf. Definitions \ref{def:inflByExt} and \ref{defMeshInflGeo} respectively. Then, $\mathcal{M}^*_{\Gamma}(\Omega)$ and $\mathcal{M}^*_{\Gamma}$ are equal up to a relabeling.
\end{theorem}

\begin{proof}
To fix ideas, we write the proof for the case $n=3$ (the case $n=2$ is analogous). Recall that the elements of $\mathcal{M}^*_{\Gamma}(\Omega)$ are the pairs $(F,K)$ with $K \in \mathcal{M}_\Omega$ and $F \in \mathcal{M}_\Gamma \cap \mathcal{F}(K)$. Given such a pair, we define 
$$\varphi((F,K)) \isdef [F]_{|[K]}$$ 
where $[K]$ is the natural orientation of $K$ (recall the notation from Eq. \eqref{inducedOrientFacet}). Since $\mathcal{M}_\Gamma$ is disjoint from the boundary of $\mathcal{M}_\Omega$, by \Cref{lemConsistNatural} this defines a bijection 
\[\varphi: \mathbf{K}_{\mathcal{M}^*_{\Gamma}(\Omega)} \to \mathbf{K}_{\mathcal{M}^*_{\Gamma}}\,,\]
with, obviously,  $\mathcal{K}_1 \circ \varphi = \mathcal{K}_2$, where $\mathcal{K}_1$ and $\mathcal{K}_2$ are the realizations of $\mathcal{M}^*_{\Gamma}$ and $\mathcal{M}^*_{\Gamma}(\Omega)$, respectively.

It remains to show that the adjacency graphs of $\mathcal{M}^*_{\Gamma}$ and $\mathcal{M}^*_{\Gamma}(\Omega)$ are compatible with this bijection. This amounts to proving that, for $(F,K) \in \mathbf{K}_{\mathcal{M}^*_{\Gamma}(\Omega)}$ and $S \in \mathcal{F}(F)$, there holds
\begin{equation}
\label{claimCompatGraphs}
\varphi\left(\mathcal{N}_b((F,K),S)\right) = \widetilde{\mathcal{N}}_b\left(\varphi((F,K)),S\right)\,.
\end{equation}
Hence, pick a pair $(F,K) \in \mathcal{M}_\Omega \times \mathcal{M}_\Gamma$, with $F \in \mathcal{F}(K)$, let $[F] = \varphi((F,K))$ and let $S \in \mathcal{F}(F)$. Let $\mathcal{M}^*_{\Omega \setminus \Gamma}$ be the fractured mesh defined by $\mathcal{M}_\Omega$ and $\mathcal{M}_\Gamma$ (cf \Cref{def:fracturedMesh}). By definition of $\mathcal{M}^*_{\Omega \setminus \Gamma}$, we have
\[\perp \overunderset{F}{\mathcal{M}^*_{\Omega \setminus \Gamma}}{\longleftrightarrow} K\,.\] 
According to \Cref{lemChain}, we can introduce the chain 
\begin{equation}
\label{chainProof}
\perp \overunderset{F}{\mathcal{M}^*_{\Omega \setminus \Gamma}}{\longleftrightarrow} K \overunderset{F_1}{\mathcal{M}^*_{\Omega \setminus \Gamma}}{\longleftrightarrow} K_2 \overunderset{F_2}{\mathcal{M}^*_{\Omega \setminus \Gamma}}{\longleftrightarrow} \,\ldots\, \overunderset{F_{Q-2}}{\mathcal{M}^*_{\Omega \setminus \Gamma}}{\longleftrightarrow} K_{Q-1} \overunderset{F_{Q-1}}{\mathcal{M}^*_{\Omega \setminus \Gamma}}{\longleftrightarrow} K_{Q}\overunderset{F_{Q}}{\mathcal{M}^*_{\Omega \setminus \Gamma}}{\longleftrightarrow}\,\perp\,.
\end{equation}
Note that since $S \in \mathcal{F}(F_Q)$, $S \in \sigma(\mathcal{M}_\Gamma)$ and $\mathcal{M}_\Gamma \cap \partial \mathcal{M}_\Omega = \emptyset$, it must be true that $F_Q \in \mathcal{M}_\Gamma$. We rewrite $K' \isdef K_{Q}$, $F' \isdef F_Q$ so that
$$(F',K') = \mathcal{N}_b((F,K),S)\,.$$ 
From \Cref{corObvious}, we deduce that for each $i \in \{1\,,\ldots\,,Q\}$,
\[F_i = \arg\min \enstq{\angle([F],\tilde{F})}{\tilde{F} \in \mathcal{F}(\mathcal{M}_\Omega) \setminus \{F_j\}_{1 \leq j \leq i-1}}\,.\]
By definition of $\mathcal{M}^{*}_{\Omega \setminus \Gamma}$, for $i \in \{1\,,\ldots\,,Q-1\}$, the facet $F_i$ is  not in $\mathcal{M}_\Gamma$. Hence
\begin{equation}
\label{FQminAngle}
F_Q = F' = \arg\min \enstq{\angle([F],\tilde{F})}{\tilde{F} \in \mathcal{M}_\Gamma}\,.
\end{equation}
Finally, let 
\[[F'] \isdef \varphi((F',K'))\,.\]
Reasoning as in the proof of \Cref{lemInducedOrientBound}, we see that $[F]$ and $[F']$ are consistently oriented. From this property, Eq.~\eqref{FQminAngle} and \Cref{defMeshInflGeo}, it follows that
\[\widetilde{N}_b([F],S) = [F']\,,\]
which proves the claim \eqref{claimCompatGraphs} and concludes the proof of the theorem. 
\end{proof}

The following result is immediate, by the very definition of $\widetilde{\mathcal{N}}_b$:
\begin{lemma}
Let $\mathcal{M}_\Gamma$ be a $(n-1)$-dimensional generalized mesh in $\R^{n}$, with $n = 2$ or $3$. Then $\mathcal{M}^*_\Gamma$ is orientable, with a compatible orientation given by
\[\forall [F] \in \mathcal{M}^*_\Gamma\,, \quad  [F]_{\mathcal{M}^*_\Gamma} \isdef [F]\,.\]
\end{lemma}

%\begin{figure}
%	\centering
%	\includegraphics[width=0.7\textwidth]{../figures/counterEx}
%	\caption{Two fracked triangular meshes with a common $\mathcal{M}_\Gamma$ but distinct $\mathcal{M}^*_\Gamma$. \toDo{Must do a better figure, but this is for discussion only.}}
%	\label{counterExInflEdg}
%\end{figure}

\section{Finite element exterior calculus on generalized meshes}

\label{sec:FEEC}

\subsection{Whitney forms}

We now fix the Euclidean space $\R^m$ as the ambient space. Every generalized mesh discussed below has its vertices in $\R^m$ and its elements are non-degenerate $n$-simplices, with $n \leq m$. Recall that for an element $\mathbf{k}$ of $\mathcal{M}^*$ attached to the simplex $K$, we write $\abs{\mathbf{k}} \isdef \abs{K}$.

In this setting, we discuss the construction of lowest-order discrete differential forms on generalized meshes, which are the simplest specimen of trial and test spaces required for Finite-Element Exterior Calculus (FEEC, see \cite{FEEC}, see also \cite{boon2021functional} for a related work on differential forms on mixed-dimensional geometries). Those spaces of discrete differential forms are spanned by locally supported basis functions, known as Whitney forms \cite{whitney}.

To define Whitney forms we need some additional structure. We have to choose an orientation for every subsimplex $S \in \sigma(\mathcal{M}^*)$. One standard way to do this is to choose an arbitrary order on the finite set $\sigma_0(\mathcal{M}^*)$ and equip each $d$-simplex with the orientation corresponding to this order. 
Typically, this doesn't incur any additional cost in the implementation, as simplices are stored using arrays, which are naturally ordered. 

We also need {\em barycentric coordinate functions} on a non-degenerate $n$-simplex $K$. Given a vertex $V$ of $K$, the barycentric coordinate $\lambda^{K}_V: \R^{n} \to \R$ is the affine function defined by the equations
\begin{equation}
\label{eq:defBarycentricFun}
\forall V' \in \sigma_0(K)\,, \quad \lambda_V^{K}(V') = \begin{cases}
1 & \textup{if } V = V'\,,\\
0 & \textup{otherwise.}
\end{cases}
\end{equation}

\begin{mdframed}
\begin{definition}[Whitney form associated to a generalized facet]
\label{def:WhitneyForms}
Consider a generalized $d$-subfacet $\mathbf{s} \in \mathbf{S}_d(\mathcal{M}^*)$, $\mathbf{s} = (S,\gamma)$ with the orientation of $S$ given by the ordering $(V_1,\ldots,V_{d+1})$. The associated {\em Whitney $d$-form} $\omega_{\mathbf{s}}$ is a tuple of differential forms
\[\omega_{\mathbf{s}} = (\omega_{\mathbf{s}}^{\mathbf{k}})_{\mathbf{k} \in \mathcal{M}^*}\,.\]
For each $\mathbf{k} \in \mathcal{M}^*$, $\omega_{\mathbf{s}}^{\mathbf{k}}$ is the $d$-differential form on $\abs{\mathbf{k}}$ defined by 
\begin{equation}
\label{whitneyElem}
\omega^{\mathbf{k}}_{\mathbf{s}} \isdef \begin{cases}
	{\displaystyle\sum_{j = 1}^{d+1}(-1)^{j+1} \lambda_{V_j}^{K} \wedge \textup{d}\lambda_{V_1}^{K} \wedge \ldots\wedge\widehat{\textup{d}\lambda_{V_j}^{K}} \wedge \ldots \wedge \textup{d}\lambda_{V_{d+1}}^{K}\,,} & \textup{for } \mathbf{k} \in \gamma\,, \\[1.5em]
	0 & \textup{for } \mathbf{k} \notin \gamma\,, 
\end{cases} 
\end{equation}
where $K$ is the simplex attached to $\mathbf{k}$, $\textup{d}$ designates the exterior derivative, and the hat notation is used to denote a suppressed term.
\end{definition}
\end{mdframed}
It is a standard fact that this definition is correct, i.e. that the formula \eqref{whitneyElem} is invariant with respect to even permutations of the chosen ordering $(V_1,\ldots,V_{d+1})$. Note that for $d = 0$, when $S$ is a vertex, say, $V_1$ of $\mathbf{k} \in \gamma$, then $\omega^{\mathbf{k}}_{\mathbf{s}}$ agrees with the barycentric coordinate function of $\abs{\mathbf{k}}$ associated to the vertex $V_1$ for $\mathbf{k}$.

Using the vector space structure of tuples of differential forms, we can consider linear combinations of the Whitney forms defined above, and we define $\Lambda^d(\mathcal{M}^*)$ as the vector space spanned by $\{\omega_\mathbf{s}\}_{\mathbf{s} \in \mathbf{S}_d(\mathcal{M}^*)}$. We call its elements the {\em Whitney $d$-forms on $\mathcal{M}^*$}.
\begin{mdframed}
\begin{definition}[Trace]
\label{def:trace}
Given a Whitney $d$-form $\omega = (\omega^{\mathbf{k}})_{\mathbf{k}\in {\mathcal{M}^*}} \in \Lambda^d(\mathcal{M}^*)$, the trace $\textup{Tr}\,\omega$ is the Whitney $d$-form $\nu = (\nu^\mathbf{f})_{\mathbf{f} \in {\partial \mathcal{M*}}} \in \Lambda^d(\partial \mathcal{M}^*)$ where, for each element $\mathbf{f} = (F,\mathbf{k})$ of $\partial\mathcal{M}^*$, $\nu^{\mathbf{f}}$ is given on $\abs{F}$ by
\[ \nu^{\mathbf{f}} \isdef \textup{Tr}_{\abs{\mathbf{k}},\abs{F}}\omega^{\mathbf{k}}\,.\]
Here, $\textup{Tr}_{\Omega,\Omega'} \mu$ is the {\em trace} of the differential form $\mu$ from the manifold $\Omega$ to the submanifold $\Omega' \subset \Omega$ \cite[p.16]{FEEC}. It is defined as the pullback
\[\textup{Tr}_{\Omega,\Omega'} \mu \isdef \iota^* \mu\]
where $\iota$ is the inclusion $\Omega' \hookrightarrow \Omega$. 
\end{definition}
\end{mdframed}
\begin{lemma}
\label{patchCond}
The Whitney $d$-forms satisfy the following {\em patch condition}: if  $\mathbf{k} \overunderset{F}{\mathcal{M}^*}{\longleftrightarrow} \mathbf{k}'$, then
\[ \forall \omega \in \Lambda^d(\mathcal{M}^*)\,, \quad \textup{Tr}_{\abs{\mathbf{k}},\abs{F}} \omega^{\mathbf{k}}=\textup{Tr}_{\abs{\mathbf{k'}},\abs{F}} \omega^{\mathbf{k'}} \,.\]
\end{lemma}
\begin{proof}
By linearity, it suffices to check that the patch condition is satisfied by $\omega_{\mathbf{s}}$ for each $\mathbf{s} \in \mathbf{S}_d(\mathcal{M}^*)$. Hence let $\mathbf{s} = (S,\gamma)$, with the orientation defined by the ordering $(V_1\,,\ldots\,,V_{d+1})$. 
We first remark that if $K$ is the simplex attached to $\mathbf{k}$, and if $S \in \sigma(K)$ doesn't contain $V_j$, the function $\lambda_{V_j}^{K}$ is identically $0$ on $\abs{S}$. Hence, if $F$ is a facet of $\mathbf{k}$ not containing $S$, then $\omega_{\mathbf{s}}^{\mathbf{k}}$ vanishes on $\abs{F}$. Consequently, if $\mathbf{k}$ and $\mathbf{k}'$ are adjacent through such a facet $F$, the patch condition is immediately verified. On the other hand, if $F$ does contain $S$, then it follows that $\mathbf{k}$ and $\mathbf{k}'$ are adjacent through $S$, so there are two cases: either $\mathbf{k} \notin \gamma$ and $\mathbf{k'} \notin \gamma$, either $\mathbf{k} \in \gamma$ and $\mathbf{k'} \in \gamma$. The patch condition is obvious in the former case since $\omega_\mathbf{s}$ then vanishes on both $\mathbf{k}$ and $\mathbf{k'}$. Finally the patch condition in the latter case follows from
\begin{itemize}
\item[1.] the property that if $K$ is an $n$-simplex and $F$ one of its facets, then 
\begin{equation}
\label{propertyBaryCoord}
\forall V \in \sigma_0(F)\,, \quad \lambda_V^{F} = \textup{Tr}_{\abs{K},\abs{F}}\,\lambda_V^K\,,
\end{equation}
\item[2.] the commutation of pullbacks with wedge products and exterior derivative. \qedhere
\end{itemize} 
\end{proof}
\begin{remark}
In particular, if $\mathcal{M}^*_{\Omega \setminus \Gamma}$ is a fractured mesh of $\Omega \setminus \Gamma$, in the notation of \Cref{sec:fracturedMeshes}, then one can deduce from \autoref{patchCond} that
\[\Lambda^0(\mathcal{M}^*_{\Omega \setminus \Gamma}) \simeq \enstq{u \in W^1(\Omega \setminus \Gamma)}{u_{|K} \textup{ is affine for every } K \in \mathcal{M}_\Omega}\,,\]
where the $\simeq$ sign denotes an explicit, natural bijection and $W^1(U)$ is the Sobolev space of square-integrable functions in the open set $U$ with a square-integrable weak derivative, see e.g. \cite[p. 73]{mclean}.
\end{remark}

\subsection{Surjectivity of the trace operator}
\label{sec:traceSurject}

In this section, we investigate the surjectivity of $\textup{Tr}: \Lambda^d(\mathcal{M}^*) \to \Lambda^d(\partial \mathcal{M}^*)$. As mentioned previously, this is a key property needed for the boundary element applications \cite{averseng2022ddm}. 

A triangular mesh $\mathcal{M}$ is {\em edge-connected} if for any $T,T' \in \mathcal{M}$, one can link $T$ and $T'$ by a chain of triangles in $\mathcal{M}$, such that two consecutive triangles have a common edge. A triangular mesh $\mathcal{M}$ has a {\em point contact} if there exists a vertex $S \in \mathcal{M}_\Gamma$ such that $\textup{st}(S,\mathcal{M})$ is not edge-connected. An example of point-contact is shown in \Cref{fig:pointContact}. The main result of this section is the following.
\begin{theorem}[Surjectivity of the trace operator]
\label{thmSurject}
Let $\mathcal{M}_\Omega$ be a regular tetrahedral mesh in $\R^3$, and let $$\mathcal{M}_\Gamma \subset \mathcal{F}(\mathcal{M}_\Omega)\setminus \partial \mathcal{M}_\Omega\,.$$
Let $\mathcal{M}^*_{\Omega \setminus \Gamma}$ be the fractured mesh defined in \Cref{def:fracturedMesh}. Then, for $d \in \{1,2\}$, 
\[\textup{Tr}\left(\Lambda^d(\mathcal{M}^*_{\Omega \setminus \Gamma})\right) = \Lambda^d(\partial \mathcal{M}^*_{\Omega \setminus \Gamma})\,.\]
For $d =0$, the previous equality holds if and only if $\mathcal{M}_\Gamma$ has no point contacts.
\end{theorem}
\begin{figure}[H]
\centering
\includegraphics[width=0.3\textwidth]{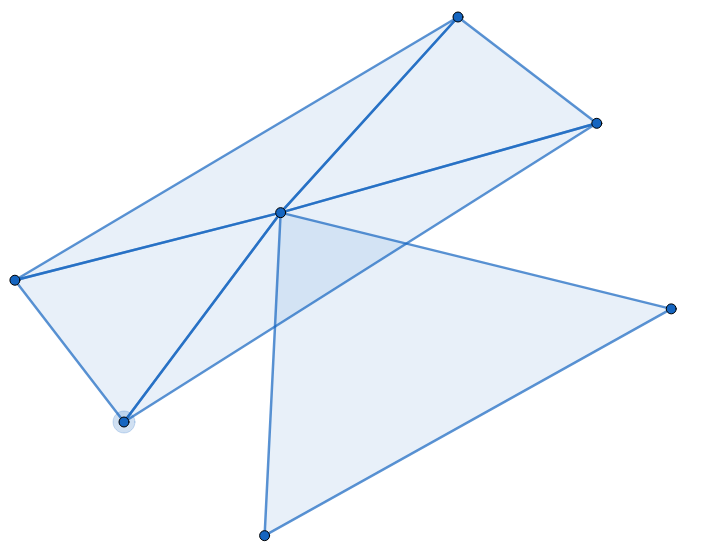}
\caption{Example of a point contact.}
\label{fig:pointContact}
\end{figure}
\begin{remark}
In dimension $2$, i.e. if $\mathcal{M}_\Omega$ is a regular triangular mesh in $\R^2$, then the results below show that the surjectivity holds without any conditions on $\mathcal{M}_\Gamma$.  
\end{remark}

To prove \Cref{thmSurject}, we start by deriving a local criterion. In what follows, $\mathcal{M}^*$ is a $n$-dimensional gen-mesh and $d \in \{0,\ldots,n-1\}$. We consider a generalized facet $\mathbf{s} = (S,\gamma)$ attached to a boundary subsimplex $S \in \sigma_d(\partial \mathcal{M}^*)$. The  {\em projected star} of $\mathbf{s}$ is defined by
\[\textup{st}_\partial\, \mathbf{s} \isdef \enstq{\mathbf{f} = (F,\mathbf{k}) \in {\partial\mathcal{M}^*}}{S \in \sigma(F) \textup{ and } \mathbf{k} \in \gamma}\,.\] 
We say that $\textup{st}_\partial \,\mathbf{s}$ is {\em connected through $S$} if, for any two elements $\mathbf{f},\mathbf{f}' \in \textup{st}_\partial \mathbf{s}$, there exists a chain 
\[\mathbf{f} = \mathbf{f}_1 \overunderset{S}{\partial \mathcal{M}^*}{\longleftrightarrow} \mathbf{f}_2  \overunderset{S}{\partial \mathcal{M}^*}{\longleftrightarrow}  \ldots  \overunderset{S}{\partial \mathcal{M}^*}{\longleftrightarrow} \mathbf{f}_{Q-1}  \overunderset{S}{\partial \mathcal{M}^*}{\longleftrightarrow} \mathbf{f}_{Q} = \mathbf{f}'\,.\]
with $\mathbf{f}_1\,,\ldots\,,\mathbf{f}_Q \in \textup{st}_\partial\, \mathbf{s}$. 
\begin{lemma}[Local criterion]
\label{lemAuxTrace}
Suppose that the projected star of $\mathbf{s}$ is connected through $S$. Then, there exists a generalized $d$-subfacet $\mathbf{t} \in \mathbf{S}_d(\partial \mathcal{M}^*)$ given by 
$$\mathbf{t} = (S,\textup{st}_\partial\, \mathbf{s})\,.$$
Moreover, the Whitney forms $\omega_{\mathbf{s}} \in \Lambda^d(\mathcal{M}^*)$ and $\omega_{\mathbf{t}} \in \Lambda^d(\partial \mathcal{M}^*)$ are related by
\[\omega_{\mathbf{t}} = \pm \textup{Tr}\,\omega_{\mathbf{s}}\,.\]
A positive sign occurs if and only if the orientations of the $d$-subsimplex $S$ in $\mathcal{M}^*$ and $\partial \mathcal{M}^*$ agree.
\end{lemma}
\begin{proof}
Pick an element $\mathbf{f}_0 = (F_0,\mathbf{k}_0) \in \textup{st}_\partial\, \mathbf{s}$, and let $\mathbf{t} = (S,\eta) \in \mathbf{S}_d(\partial \mathcal{M}^*)$ be the unique generalized $d$-subfacet attached to $S$ such that $\mathbf{f}_0 \in \eta$. We claim that 
\begin{equation}
\label{etaEqDs}
\eta = \textup{st}_\partial\, \mathbf{s}\,.
\end{equation}
Indeed, if $\mathbf{f}'$ is another element of $\textup{st}_\partial\, \mathbf{s}$, then by the assumption of the lemma, $\mathbf{f}_0$ and $\mathbf{f}'$ are in the same component of the graph $\mathcal{G}_{\partial \mathcal{M}^*}(S)$, so $\mathbf{f}' \in \eta$. Hence $\textup{st}_\partial\, \mathbf{s} \subset \eta$. 

Conversely, if $\mathbf{f}' = (F',\mathbf{k}') \in \eta$, then $\mathbf{f}'$ is in the same connected component as $\mathbf{f}$ in the graph $\mathcal{G}_{\partial\mathcal{M}^*}(S)$, which by definition of adjacency of the boundary, implies that $\mathbf{k'}$ is in the same component as $\mathbf{k}_0$ in the graph $\mathcal{G}_{\mathcal{M}^*}(S)$. Hence $\mathbf{k'} \in \gamma$, so $\mathbf{f}' \in \textup{st}_\partial\, \mathbf{s}$. Therefore, $\eta \subset\textup{st}_\partial\, \mathbf{s}$, which proves \eqref{etaEqDs}. 

Let $\mu = \textup{Tr}\,\omega_{\mathbf{s}}$. For each $\mathbf{f} = (F,\mathbf{k}) \in \partial \mathcal{M}^*$, we compare the expressions of $\mathcal{\mu}^\mathbf{f}$ and $\omega_{\mathbf{t}}^{\bf f}$. 
On the one hand, if $\mathbf{f} \in \textup{st}_\partial\, \mathbf{s}$, then, writing $S = (V_1\,,\ldots\,,V_{d+1})$ (with an order defining the orientation of $S$ in $\mathcal{M}^*$), and denoting by $K$ the simplex attached to $\mathbf{k}$, 
\[\begin{split}
\mu^\mathbf{f} &= \textup{Tr}_{\abs{K},\abs{F}} \left(\sum_{j = 1}^{d+1} (-1)^{j+1} \lambda_{V_j}^{K} \wedge \textup{d} \lambda_{V_1}^K\ldots \wedge \widehat{\textup{d} \lambda_{V_j}^K} \wedge \ldots \wedge \textup{d} \lambda_{V_{d+1}}^K\right) \\
& = \sum_{j = 1}^{d+1} (-1)^{j+1} \lambda_{V_j}^{F} \wedge \textup{d} \lambda_{V_1}^F\ldots \wedge \widehat{\textup{d} \lambda_{V_j}^F} \wedge \ldots \wedge \textup{d} \lambda_{V_{d+1}}^F
\end{split}\]
again by property \eqref{propertyBaryCoord}. Clearly, this is equal to $\omega_{\mathbf{t}}^{\mathbf{f}}$ on $\abs{F}$, up to a sign (the same sign will occur if the orientation of $S$ in $\mathcal{M}^*$ agrees with the one in $\partial \mathcal{M}^*$). 

On the other hand, if $\mathbf{f} \notin \textup{st}_\partial\, \mathbf{s}$, then $\omega_{\mathbf{t}}^{\mathbf{f}}$ vanishes on $\abs{F}$, so it remains to check that the same holds for $\mu^\mathbf{f}$. By definition of $\textup{st}_\partial\, \mathbf{s}$, either $\mathbf{k} \notin \gamma$, or $S \notin \sigma(F)$. In the former case, we have $\mu^\mathbf{f} = 0$ since $\omega_{\mathbf{s}}^{\mathbf{k}} = 0$. In the latter, there is a vertex $V$ of $S$ not in $F$, so $\lambda_V^K = 0$ identically on $\abs{F}$. This confirms that $\mu^{\mathbf{f}}$ vanishes on $\abs{F}$, and concludes the proof of the lemma. 
\end{proof}
\begin{corollary}
\label{corTrace}
Suppose that for each $d$-subsimplex $S \in \sigma_d(\partial \mathcal{M}^*)$ and for each generalized $d$-subfacet $\mathbf{s}\in \mathbf{S}_d(\mathcal{M}^*)$ attached to $S$, $\textup{st}_\partial\, \mathbf{s}$ is connected through $S$. Then, there holds
\[\textup{Tr}\left(\Lambda^d(\mathcal{M}^*)\right) = \Lambda^d(\partial \mathcal{M}^*)\,.\]
\end{corollary}
\begin{proof}
Let $\mathbf{t} = (S,\eta) \in \mathbf{S}_d(\partial \mathcal{M}^*)$, pick $\mathbf{f}_0 = (F_0,\mathbf{k}_0) \in \eta$, and let $\mathbf{s} = (S,\gamma) \in \mathbf{S}_d(\mathcal{M}^*)$ be the generalized $d$-subfacet of $\mathcal{M}^*$ defined by the condition that $\mathbf{k}_0 \in \gamma$. We have $\mathbf{f}_0 \in \textup{st}_\partial\, \mathbf{s}$, so, following the beginning of the proof of \Cref{lemAuxTrace}, in fact $\mathbf{t} = (S,\textup{st}_\partial\, \mathbf{s})$. Hence, by \Cref{lemAuxTrace}, there holds
\[\omega_{\mathbf{t}} = \pm \textup{Tr}\, \omega_{\mathbf{s}}\,.\]
Since this holds for any $\mathbf{t} \in \mathbf{S}_d(\partial\mathcal{M}^*)$, the conclusion follows immediately.
\end{proof}
It is not difficult to check that the condition of \Cref{lemAuxTrace} is always satisfied for $d \geq n-2$, therefore:
\begin{corollary}
\label{corTraceDNminus2}
Let $\mathcal{M}^*$ be a $n$-dimensional gen-mesh with $n \geq 2$. Then there holds
\[\textup{Tr}\left(\Lambda^{n-1}(\mathcal{M}^*)\right) = \Lambda^{n-1}(\partial \mathcal{M}^*)\,, \quad \textup{Tr}\left(\Lambda^{n-2}(\mathcal{M}^*)\right) = \Lambda^{n-2}(\partial \mathcal{M}^*)\,.\] 
\end{corollary}

Let $\mathcal{M}_\Omega$ be a regular tetrahedral mesh in $\R^3$. As in Section \ref{sec:fracturedMeshes}, let $\mathcal{M}_\Gamma \subset \mathcal{F}(\mathcal{M}_\Omega) \setminus \partial \mathcal{M}_\Omega$ and let $\mathcal{M}^*_{\Omega \setminus \Gamma}$ be the fractured mesh defined in \Cref{def:fracturedMesh}. By \Cref{corTraceDNminus2}, $\textup{Tr}:\Lambda^d(\mathcal{M}^*_{\Omega \setminus \Gamma}) \to \Lambda^d(\partial \mathcal{M}^*_{\Omega \setminus \Gamma})$ is surjective for $d = 1,2$. To establish \Cref{thmSurject}, it remains to show the following result. The proof is deferred to \ref{app:proofSurject}.

%let $\mathcal{M}_{\Gamma}$ be a triangular mesh and  $\mathcal{M}_\Omega$, a regular tetrahedral mesh such that $\mathcal{M}_\Gamma \subset \mathcal{F}(\mathcal{M}_\Omega) \setminus \partial \mathcal{M}_\Omega$. Correspondingly, let $\mathcal{M}^*_{\Omega \setminus \Gamma}$, $\mathcal{M}^*_\Gamma(\Omega)$ and $\mathcal{M}^*_\Gamma$ be the fractured mesh (cf. \Cref{def:fracturedMesh}), the extrinsic inflation via $\mathcal{M}_\Omega$ (cf. \Cref{def:inflByExt}) and the intrinsic inflation of $\mathcal{M}_\Gamma$ (cf.  \Cref{defMeshInflGeo}), respectively. Since $\mathcal{M}^*_\Gamma$ is a relabeling of $\mathcal{M}^*_\Gamma(\Omega)$ by \Cref{theoremInflation}, there exists a bijection 
%\[\varphi_\Gamma: \mathbf{K}_{\mathcal{M}^*_{\Gamma}} \to \mathbf{K}_{\mathcal{M}^*_\Gamma(\Omega)}\]
%satisfying the conditions of \Cref{defRelabeling}. Noting that $\mathbf{K}_{\mathcal{M}^*_\Gamma(\Omega)}\subset \mathbf{K}_{\partial \mathcal{M}^*_{\Omega \setminus \Gamma}}$, this allows to define a ``pullback" 
%\[\begin{array}{ccc}
%	\varphi^*_{\Gamma}: \quad \Lambda^d(\partial \mathcal{M}^*_{\Omega \setminus \Gamma}) &\to& \Lambda^d(\mathcal{M}^*_{\Gamma})\\
%	\omega  &\mapsto& \mu\,,
%\end{array}
%\]
%where $\mu$ is defined by 
%\[\mu^{\mathbf{f}} \isdef \omega^{\mathbf{f'}}\,, \quad \mathbf{f} \in \mathcal{M}^*_\Gamma\,,\,\, \mathbf{f'} \isdef \varphi_\Gamma(\mathbf{f})\,.\]
%Essentially, $\varphi^*_\Gamma$ is the restriction of Whitney $d$-forms from $\partial \mathcal{M}^*_{\Omega \setminus \Gamma}$ to $\mathcal{M}^*_\Gamma$. 
\begin{lemma}
\label{lemMainSurject}
The following conditions are equivalent:
\begin{itemize}
\item[(i)] For each vertex $S$ of $\mathcal{M}_\Gamma$, $\textup{st}(S,\mathcal{M}_\Gamma)$ is {\em edge-connected}.
\item[(ii)] $\textup{Tr}\left(\Lambda^0(\mathcal{M}^*_{\Omega \setminus \Gamma})\right) = \Lambda^0(\partial \mathcal{M}^*_{\Omega \setminus \Gamma})$.
\end{itemize}
\end{lemma}
\begin{remark}
\label{rem:pointContact}
The issue, when there is a point contact at $S$, is that the star of $S$ has several distinct edge-connected components. The Whitney $0$-forms on the boundary can be defined independently on each of those components, while the Whitney forms in the volume must satisfy a compatibility condition at $S$. This is why the surjectivity fails in this case.
\end{remark}

\subsection{Finite element assembly}

Given an $n$-dimensional generalized mesh $\mathcal{M}^*$ and $0 \leq d \leq n$, we discuss the computation of the Galerkin matrix associated to a bilinear form
$$a: \Lambda^d(\mathcal{M}^*) \times \Lambda^d(\mathcal{M}^*)\to \R\,.$$
We assume that $a$ has the form
\[a(\omega,\omega') = \sum_{\mathbf{k} \in \mathcal{M}^*} a_{\mathbf{k}}(\omega_{\mathbf{k}},\omega'_{\mathbf{k}})\,,\]
where for each $\mathbf{k} \in {\mathcal{M}^*}$, $a_{\mathbf{k}}$ is a bilinear form on the space of $d$-differential forms on $\abs{\mathbf{k}}$.  
The Galerkin matrix $\mathbf{A}$ of $a$ is defined by
\[\mathbf{A}_{\mathbf{s},\mathbf{t}} = a(\omega_\mathbf{s},\omega_{\mathbf{t}})\,,\quad \mathbf{s},\mathbf{t} \in \mathbf{S}_d(\mathcal{M}^*)\,,\]
where $\{\omega_{\mathbf{s}}\}_{\mathbf{s} \in \mathbf{S}_q(\mathcal{M}^*)}$ is the Whitney basis functions defined in the previous section. We use indexing by generalized $d$-subfacets to avoid the heavier notation that would result from introducing orderings.

As in standard finite element computations, one first needs a method to compute so-called {\em local matrices}, involving the local Whitney forms
defined on each element. For each $d$-subsimplex $S$ of an element $\mathbf{k} \in \mathcal{M}^*$, there is a unique generalized $d$-subfacet $\mathbf{s}$ of the form
\[\mathbf{s} = (S,\gamma) \textup{ with } \mathbf{k} \in \gamma\,.\]
We write $\mathbf{s} =:J(\mathbf{k},S)$. In the implementation, $J$ corresponds to a ``local-to-global" index mapping. 
One can then form the local $\binom{n}{d}\times \binom{n}{d}$ matrices $A_{loc}(\mathbf{k})$, defined by
\[\left[A_{loc}(\mathbf{k})\right]_{S,S'} \isdef a(\omega_{\mathbf{s}}^{\mathbf{k}},\omega_{\mathbf{s}'}^{\mathbf{k}})\quad S,S' \in \sigma_d(\mathbf{k})\,, \mathbf{s} = J(\mathbf{k},S)\,, \mathbf{s}' = J(\mathbf{k},S')\,.\]

To assemble the global matrix $\mathbf{A}$ from the local matrices $A_{loc}$, one can then use \Cref{algo:assembly}. In the context of boundary elements, the bilinear form $a$ rather takes the form
\[a(\omega,\omega') = \sum_{\mathbf{k} \in \mathbf{K}_{\mathcal{M}^*}}\sum_{\mathbf{k}'\in \mathbf{K}_{\mathcal{M}^*}} a_{\mathbf{k},\mathbf{k}'}(\omega_{\mathbf{k}},\omega_{\mathbf{k}'}')\,.\]
This case can be tackled very similarly, using two nested loops over elements, instead of just one, in \Cref{algo:assembly}. 

\begin{algorithm}
\caption{\texttt{Assembly}$(\mathcal{M}^*,d,A_{loc})$}
\label{algo:assembly}
\mbox{\bf INPUTS:} Generalized mesh $\mathcal{M}^*$, subsimplex dimension $d$, local matrices $A_{loc}$\\
\mbox{\bf RETURNS:} The (sparse) matrix $\mathbf{A}$. \\

$N_d \GETS \#{\mathbf{S}_d(\mathcal{M}^*)} \quad {\color{green!50!black}\texttt{\% Number of generalized $d$-subfacets}}\\
\mathbf{A} \GETS {\tt zeros}(N_d, N_d) \quad {\color{green!50!black}\texttt{\% Initialize matrix}}\\
\\
\mbox{\bf FOR } \mathbf{k} \in \mathbf{K}_{\mathcal{M}^*} \\
\begin{array}{rl}
& A \GETS A_{loc}(\mathbf{k}) \quad {\color{green!50!black}\texttt{\% Compute local matrix}}\\[1ex]
&\mbox{\bf FOR } S,S' \in \sigma_d(\mathbf{k}) \\
&\begin{array}{rl}
	& \mathbf{s} \GETS J(\mathbf{k},S)  \quad {\color{green!50!black}\texttt{\% find indices in global matrix}}\\
	& \mathbf{t} \GETS J(\mathbf{k},S') \\
	& \mathbf{A}_{\mathbf{s},\mathbf{t}} \GETS 	\mathbf{A}_{\mathbf{s},\mathbf{t}} + A(S,S')  \quad {\color{green!50!black}\texttt{\% add local contribution to the global matrix}}	
\end{array}\\
&\mbox{\bf END FOR}\\[1ex]
\end{array}\\
\mbox{\bf END FOR}$\\[1ex]
\mbox{\bf RETURN } $\mathbf{A}$
\end{algorithm}

\section{Model application: Laplace eigenvalue problem in a disk with cut radius}
\label{sec:numerics}
To conclude this paper, we present an application of generalized meshes to the resolution of the Laplace equation in a disk with cut radius, i.e. the domain
\[\Omega = B(0,1) \setminus \left([0,1)\times \{0\}\right) \subset \R^2\,.\]
Generalized meshes are perfectly suited to represent such a geometry, see e.g. \Cref{fig:diskCutRad} below. To create meshes like this one, we have implemented a function
\[\texttt{ $\mathcal{M}^*_{\Omega \setminus \Gamma}$ = fracturedMesh($\mathcal{M}_\Omega$,$\mathcal{M}_\Gamma$)}\]
which takes as an input a regular $n$-dimensional mesh $\mathcal{M}_\Omega$, a $(n-1)$-dimensional mesh $\mathcal{M}_\Gamma$ such that $\mathcal{M}_\Gamma \subset \mathcal{F}(\mathcal{M}_\Gamma)$, and returns the fractured mesh $\mathcal{M}^*_{\Omega \setminus \Gamma}$ as defined in Section \ref{sec:fracturedMeshes}. Then, the Galerkin matrices in the basis $\{\omega_{\mathbf{v}}\}_{\mathbf{v} \in \mathbf{S}_0(\mathcal{M}^*_{\Omega \setminus \Gamma})}$  of the operators needed, (i.e. the mass matrix, for the identity operator, and the stiffness matrix, for the Laplace operator) can be assembled by using \Cref{algo:assembly}.\footnote{For our Matlab implementation, we have rather adopted a global assembly algorithm as in \cite{gypsilab}, in which the nested loops can be avoided to increase the performance.}

If initially, the mesh $\mathcal{M}_\Omega$ does not resolve the fracture, in the sense that the condition ${\mathcal{M}_\Gamma \subset \mathcal{F}(\mathcal{M}_\Gamma)}$ is not fulfilled, then one may remesh the domain $\mathcal{M}_\Omega$ using constrained meshing algorithms, see e.g. \cite{berge2019voronoi}.

\begin{figure}
\centering
\includegraphics[width=0.5\textwidth]{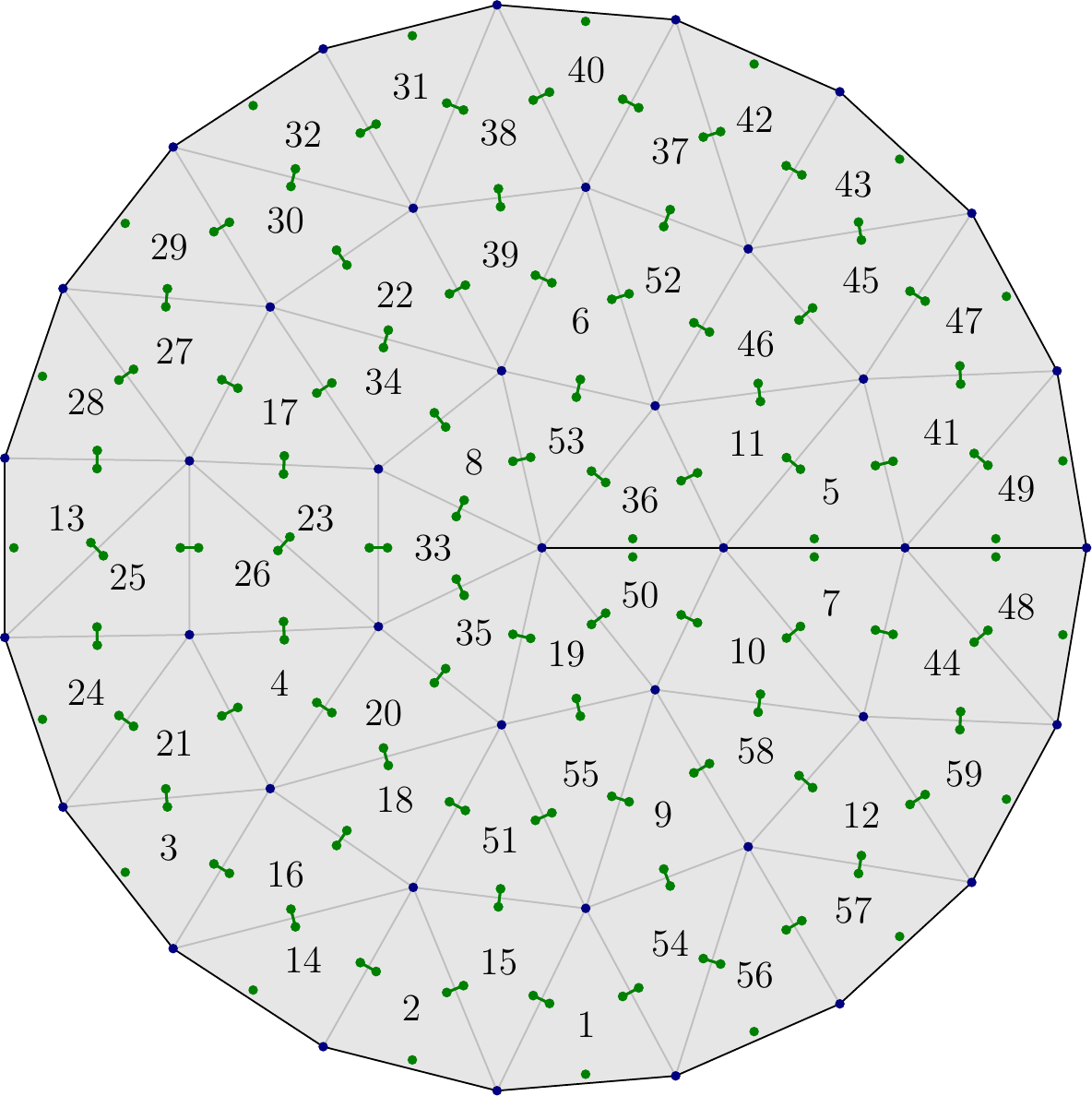}
\caption{Generalized mesh representing a disk with cut radius. The adjacency graph is represented in green, with the same conventions as in Figures \ref{crackDomainExample} and \ref{fig:duplicatedVert}}
\label{fig:diskCutRad}
\end{figure}

We consider the following eigenvalue problem:
\begin{equation}
\left\{\begin{array}{rlcll}
\Delta u &=& \lambda u &&\textup{in } \Omega\,,\\
\partial_n u &=& 0 && \textup{on } \partial \Omega\,,
\end{array}\right.
\end{equation}
whose solutions can be found analytically by separation of variables; they take the form
\[u_{n,p}(r,\theta) = f_p(r)g_n(\theta)\]
with
\begin{equation}
\label{gntheta}
g_n(\theta) = \cos( {n\theta/2})\,, \quad n \in \N\,,
\end{equation}
\begin{equation}
\label{fpr}
f_p(r) = J_{\frac{n}{2}}(\rho_{n,p}\,r)\,, \quad p \in \N\,,
\end{equation}
where $\rho_{n,p}$ is the $p$-th zero of $J_{n/2}'$, and $J_\alpha$ is the Bessel function of the first kind and order $\alpha$. The associated eigenvalue is $\lambda_{n,p} = \rho_{n,p}^2$. 

To test our implementation, we compute numerical approximations of those eigenvalues using a variational formulation and the finite element method on a generalized mesh like the one of \Cref{fig:diskCutRad}. We do not dwell on the details of the method, but refer to our Matlab implementation \cite{matlabCode}. The numerical values of the first $\lambda_{n,p}$ that we obtain for various mesh sizes are compared against high-precision reference values derived from eqs.~(\ref{gntheta}-\ref{fpr}). The numerical values indeed approach the reference values, see \Cref{table}. The fact that the first non-constant eigenfunction,
\[u_{1,1}(r,\theta) = J_{1/2}(\sqrt{\lambda_{1,1}}r) \cos(\theta/2)\] 
behaves like $O(\sqrt{r})$ near the origin, explains the slow rate of convergence for the corresponding eigenvalue, and is a manifestation of a well-known feature in the field of fracture physics (see e.g. \cite[Chap.3]{kuna}), the so-called ``crack-tip singularity". 
\begin{table}[H]
\centering
\begin{tabular}{c |c c c c}
  &$h \approx 0.7$ & $h \approx 0.35$ & $h\approx 0.18$ & $h\approx 0.09$  \\[0.5em]$\abs{\lambda_{1,h} - \lambda_{1}}/\lambda_{1}$& 0.15152& 0.062438& 0.028606& 0.014625\\
$\abs{\lambda_{2,h} - \lambda_{2}}/\lambda_{2}$& 0.046833& 0.01163& 0.0028283& 0.00083173\\
$\abs{\lambda_{3,h} - \lambda_{3}}/\lambda_{3}$& 0.054236& 0.015154& 0.0039806& 0.0011769\\
$\abs{\lambda_{4,h} - \lambda_{4}}/\lambda_{4}$& 0.076291& 0.021513& 0.0057585& 0.0016574\\
$\abs{\lambda_{5,h} - \lambda_{5}}/\lambda_{5}$& 0.094939& 0.028421& 0.007837& 0.002246
\end{tabular}
\caption{Numerical approximations of the the first eigenvalues of the Neumann eigenvalue problem on the disk with cut radius as computed by the finite element method, compared to theoretical value. We report the relative error $\abs{\lambda_{i,h} - \lambda_i}/\lambda_i$ where $\lambda_{i,h}$ is the numerical approximation returned by the finite element method and $\lambda_i$ is the corresponding reference value. The parameter $h$ is the diameter of the largest element in the mesh. The first eigenvalue, $\lambda_0 = 0$, is ignored. }
\label{table}
\end{table}

One can of course tackle more complex geometries and PDEs, use vectorial elements and work in three dimensions. Here we have restricted our attention to the simplest model problem for the sake of clarity.  

\bgroup \footnotesize
\bibliographystyle{plain}
\bibliography{bilbioGenMesh.bib}
\egroup

\appendix

\section{Proof of \Cref{lemMainSurject}}
\label{app:proofSurject}

Following \Cref{rem:pointContact}, the implication (i) $\Rightarrow$ (ii) is not difficult, so we only prove (ii) $\Rightarrow$ (i). Let $S$ be a vertex of $\sigma_0(\partial\mathcal{M}^*_{\Omega \setminus \Gamma})$. One has either $S \in \sigma_0(\mathcal{M}_\Gamma)$ or $S \in \sigma_0(\partial \mathcal{M}_\Omega)$. We restrict our attention to the first case, but the second case is similar. We shall prove that for each generalized vertex $\mathbf{s} \in \mathbf{S}_0(\mathcal{M}^*)$ attached to $S$, the projected star of $\mathbf{s}$ is connected (in the sense defined above \Cref{lemAuxTrace}). The conclusion then follows from \Cref{corTrace}.

The key idea of the proof is to consider the faces of a {\em plane graph} $G$ (for the notion of plane graph, refer to \cite[Chap. 9]{bondy1976graph}) drawn on the {\em link} of $S$. More precisely, define
$$\mathscr{S} = \textup{lk}(S,\mathcal{M}_\Omega) \isdef \partial( \textup{st}(S,\mathcal{M}_\Omega))\,,$$ 
Since $\mathcal{M}_\Omega$ is regular and $S \notin \partial \mathcal{M}_\Omega$,  $\abs{\mathscr{S}}$ is homeomorphic to a sphere \cite[Corollary 1.16]{hudson}. Some edges of $\mathscr{S}$ are incident to a triangle of $\mathcal{M}_\Gamma$: those edges define a graph $G$ drawn on $\abs{\mathscr{S}}$, as represented in \Cref{fig:maze}. Crucially, this graph is {\em connected} since $\textup{st}(S,\mathcal{M}_\Gamma)$ is  {\em edge-connected}. 

One can define a one-to-one correspondence between $\textup{st}(S,\partial \mathcal{M}^*_{\Omega \setminus \Gamma})$ -- the generalized star of $S$ defined in eq.~\eqref{eq:defGenStar} -- and the set of {\em directed edges} of $G$. For this, we consider an element $\mathbf{f} \in \textup{st}(S,\partial \mathcal{M}^*_{\Omega \setminus \Gamma})$. Let $K$ be the tetrahedron attached to $\mathbf{k}$, $T = K\setminus S$ and $E = T \cap F$. Then $E$ is an edge of $G$, and the directed edge $e$ corresponding to $\mathbf{f}$ is defined as the direction of $E$ such that $T$ lies on the right side of $e$. We will write $e = \psi(\mathbf{f})$. 

\begin{figure}[H]
\centering
\includegraphics[width=0.6\textwidth]{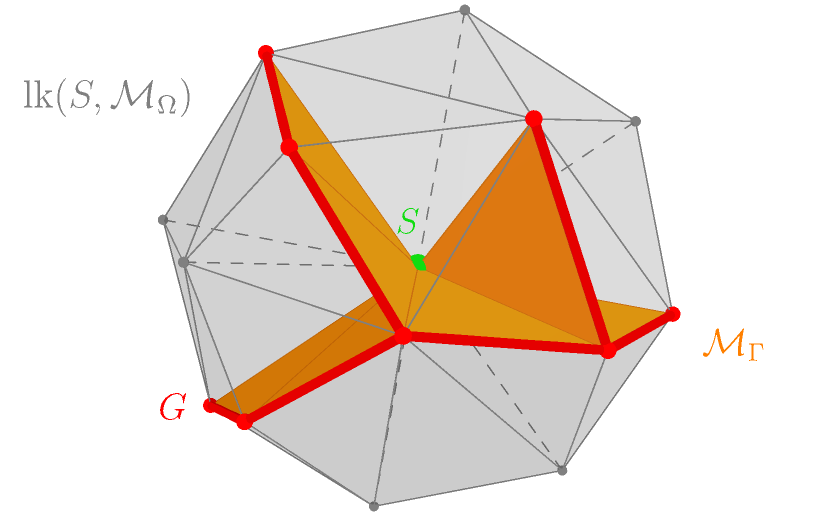}
\caption{Illustration of the idea of the proof of \Cref{thmSurject}. The link of the vertex $S$ is represented in gray, the mesh $\mathcal{M}_\Gamma$ in orange and the graph $G$ drawn on $\textup{lk}(S,\mathcal{M}_\Omega)$ in red.}
\label{fig:maze}
\end{figure}

To each generalized vertex $\mathbf{s} = (S,\gamma)\in \mathbf{S}_0(\mathcal{M}_{\Omega \setminus \Gamma}^*)$ attached to $S$, corresponds a unique {\em face} $f$ of $G$.\footnote{A face of $G$ is a connected component of $\abs{\mathscr{S}}\setminus G$, where $G$ is considered as the union of its points and edges.}. Namely, if $\mathbf{k} \in \gamma$, $K$ is the simplex attached to $K$, and $F$ is the facet of $K$ not containing $S$, then $f$ is the face of $G$ containing the interior of $\abs{F}$. If $\mathbf{f} \in \textup{st}_\partial\, \mathbf{s}$, then the edge $e = \psi(\mathbf{f})$ is on the boundary of $f$, and the face of $G$ to the right of $e$ is $f$. From now on, we fix a generalized vertex $\mathbf{s}$ attached to $S$ and let $f$ be the corresponding face of $G$. 

Starting from $\mathbf{f}_0 = (F_0,\mathbf{k}_0) \in \textup{st}_\partial\,\mathbf{s}$, we define a periodic chain of elements of $\textup{st}_\partial\,\mathbf{s}$ as follows. Let $e_0 = \psi(\mathbf{f}_0)$, and write $e_0 = (V_0,V_1)$. Next, let $\mathbf{f}_1 = \mathcal{N}_{\partial\mathcal{M}^*_{\Omega \setminus \Gamma}}(\mathbf{f}_0,SV_1)$ be the neighbor of $\mathbf{f}_0$ through the edge $SV_1$. Similarly, let $e_1 = \psi(\mathbf{f}_1)$, write $e_1 = (V_1,V_2)$ and put $\mathbf{f}_2 = \mathcal{N}_{\partial\mathcal{M}^*_{\Omega \setminus \Gamma}}(\mathbf{f}_1,SV_2)$. Using the fact that $ \partial \mathcal{M}^*_{\Omega \setminus \Gamma}$ has no boundary, this process can be repeated indefinitely and generates a periodic sequence $(\mathbf{f}_n)_{n \in \N}$ of elements of $\textup{st}_{\partial}\,\mathbf{s}$. 

Clearly, it suffices to show that every element of $\textup{st}_\partial\,\mathbf{s}$ is visited by this sequence. Hence, we pick an arbitrary element $\mathbf{f}^* = (F^*,\mathbf{k}^*) \in \textup{st}_\partial \,\mathbf{s}$, and let $e^*$ be the directed edge corresponding to $\mathbf{f}^*$, i.e. $e^* = \psi(\mathbf{f})$. 

The edges $e_n$ are in the boundary of $f$, so they define a periodic {\em boundary walk} of $f$. By the Jordan curve theorem, it holds that $e^*$ or $\overline{e^*}$ (the same edge as $e^*$ but with the opposite direction) must be visited by this walk.\footnote{Proving this rigorously can be tedious, but this is a well-known fact in planar graph theory, see e.g. \cite[p.140]{bondy1976graph}.}

We claim that in fact, $e^*$ is visited. Indeed, if $\overline{e^*}$ is visited, then there exist $\mathbf{f},\mathbf{f'} \in \textup{st}_\partial \,\mathbf{s}$ such that $\psi(\mathbf{f}) = e^*$, $\psi(\mathbf{f}') = \overline{e^*}$. Denoting the common undirected edge by $E$, we deduce that the face $f$ lies on both sides of $E$. This means that $E$ is an acyclic edge of $G$, and in this case, the boundary walk visits $E$ exactly twice, once in each direction. Therefore, $e^*$ is also visited. 

This proves that $\mathbf{f}^* = \mathbf{f}_n$ for some $n \in \N$, which concludes the proof.  \qed

\paragraph{Acknowledgement}
MA gratefully acknowledges that the key idea in the above proof, namely, to introduce the graph $G$ shown in \Cref{fig:maze}, was found by Maxence Novel. He also thanks Nikolas Stott for their help throughout the redaction of this article.

\end{document}